\documentclass[11pt]{article}

\usepackage{lastpage,extramarks} 
\usepackage{amsmath,amsthm,amsfonts,thmtools,mathtools,amssymb,mathrsfs,graphicx,graphbox,wrapfig,enumitem} 
\usepackage{booktabs,array}
\usepackage[a4paper,margin=3cm]{geometry}
\usepackage{setspace}
\usepackage{float}
\usepackage[english]{babel}
\usepackage[x11names]{xcolor}
\usepackage{tikz}
\usetikzlibrary{angles,quotes}
\usepackage[colorlinks=true,urlcolor=.,linkcolor=.,citecolor=.]{hyperref}
\usepackage{imakeidx,hyperref}
\usepackage{pdfpages}
\usepackage[shortcuts]{extdash}
\usepackage{pgfplots}
\usepackage{subfigure}
\usepackage{multirow}

\usetikzlibrary{calc,arrows.meta, positioning, shapes.geometric}
\newcommand{\hexcoord}[3]{
	($ (#1*#2*1cm + #1*#3*0.5cm, #1*#3*0.86603cm) $) 
}
\newcommand{\tikzgraph}{%
	\tikzset{
		hexnode/.style={
			circle, draw, fill=white,
			minimum size=3.5mm*\dx,
			inner sep=0pt,
			line width=0.3pt
		},
		trianglenode/.style={
			regular polygon, 
			regular polygon sides=3, 
			draw=gray, 
			fill=white,
			minimum size=1.5mm*\dx,
			inner sep=0pt,
			line width=0.3pt,
		},
		squarenode/.style={
			regular polygon, 
			regular polygon sides=4, 
			draw=gray, 
			fill=white,
			minimum size=1.5mm*\dx,
			inner sep=0pt,
			line width=0.3pt
		},
		pentanode/.style={
			regular polygon, 
			regular polygon sides=5,
			draw=gray, 
			fill=white,
			minimum size=1.5mm*\dx,
			inner sep=0pt,
			line width=0.3pt
		},
		highlight/.style={
			circle,
			draw=myred, 
			fill=none,
			minimum size=6.5mm*\dx,
			inner sep=0pt,
			line width=1pt
		},
		directed/.style={thick,->, >= {Stealth[length=4pt,width=3pt]}},
		undirected/.style={thin,dashed},
		bidirected/.style={<->,color=myblue,thick}
	}%
}
\tikzgraph

\def\dxdefault{1.1}
\def\dx{\dxdefault}

\definecolor{myblue}{HTML}{4245A3}
\definecolor{mypastel}{HTML}{1C7AA2}
\definecolor{myorange}{HTML}{FF8726}
\definecolor{myred}{HTML}{D75341}
\definecolor{mygreen}{HTML}{6DA97A}

\theoremstyle{definition}
\newtheorem{defn}{Definition}[section]
\newtheorem{lem}[defn]{Lemma}
\newtheorem{prop}[defn]{Proposition}

\newtheorem*{conject*}{Conjecture}

\theoremstyle{plain}
\newtheorem{theorem}[defn]{Theorem}
\newtheorem*{theorem*}{Theorem}

\theoremstyle{remark}
\newtheorem{ex}[defn]{Example}
\newtheorem{rem}[defn]{Remark}
\newtheorem{openq}[defn]{Question}

\newcommand{\com}[1]{}

\newcommand{\Z}{\mathbb{Z}}

\newcommand{\R}{\mathbb{R}}

\renewcommand{\S}{\mathbb{S}}
\newcommand{\eps}{\varepsilon}
\newcommand{\abs}[1]{\left\lvert#1\right\rvert}

\newcommand{\setsep}{\ \mid \ }
\newcommand{\m}{\text{-}}

\DeclareMathOperator\id{id}
\DeclareMathOperator\Imm{Imm}

\DeclareMathOperator{\im}{im}
\DeclareMathOperator{\RP}{\mathbb{R}\text{P}}
\title{An Invariant for Triple-Point-Free Immersed Spheres}

\author{Jona Seidel\footnote{seidel@mathematik.tu-darmstadt.de,\newline ORCID: 0009-0000-1282-8209,\newline Technical University Darmstadt, Department of Mathematics, Schlossgartenstr.\ 7, 64289 Darmstadt, Germany }}
\date{}

\begin{document}
	
	\maketitle
	
	\begin{abstract}
		We define an invariant of triple-point-free immersions of $ 2 $-spheres into Euclidean $ 3 $-space, taking values in $ l^1(\Z) $. It remains unchanged under regular homotopies through such immersions.
		An explicit description of its image shows that the space of triple-point-free immersed spheres has infinitely many regular homotopy classes. Consequently, many pairs of immersed spheres can only be connected by regular homotopies that pass through triple points. 
		We represent the double points of a triple-point-free immersed sphere using a directed tree, equipped with a pair relation on the edges and an integer-valued function on the vertices. The invariant depends on this function and on the vertex indegrees.\footnote{Mathematics Subject Classification \href{https://mathscinet.ams.org/mathscinet/search/mscdoc.html?code=57R42,57R45,57M15}{57R42; 57R45, 57M15},\newline Keywords: Triple Points, Double Points, Tree, Gauss Code
	}
	\end{abstract}
	
	\newpage
	

	\section{Introduction}
	In geometric topology, a central goal is to classify geometric objects up to continuous deformation. This often involves the study of invariants. \com{Other examples: plane curves, spherical curves}A classical and well-studied setting is knot theory, which analyzes embeddings of the circle $ \S^1 $ into $ \R^3 $. Another rich domain is the study of knotted surfaces, such as embeddings of the two-sphere $ \S^2 $ into $ \R^4 $, both of which yield intricate and diverse families of inequivalent embeddings.
	
	In the space of embeddings of $ \S^2 $ into $ \R^3 $, however, every embedded sphere can be deformed to a standard round sphere of the same orientation \cite{Alexander_EmbSpheres}. Even allowing for self-intersections does not yield more components: the space of immersed spheres in $ \R^3 $ forms a single regular homotopy class \cite{Smale_SE}. This foundational result by Smale sparked interest in \textit{sphere eversions}---regular homotopies between embedded spheres of opposite orientation.
	A result by Max and Banchoff first revealed an interesting structure of self-intersections in the space of immersed spheres.
	\begin{defn}
		An $ n $\emph{-tuple point} of a map $ f:X\to Y $ between sets is a point $ y\in Y $ whose preimage $ f^{-1}(\{y\}) $ contains at least $ n $ points.
	\end{defn}
	\begin{theorem*}{\cite{MB_SE}}\label{thm_SEquad}
		Every sphere eversion has at least one quadruple point.
	\end{theorem*}
	Alternative proofs can be found in \cite{Hughes_SE,Gor_Local}.
	This shows that removing immersions with quadruple points divides the space of immersed spheres into two or more components.
	The proof in \cite{MB_SE} yields a map 
	$$ \{f\in \Imm(\S^2,\R^3) \setsep f \text{ has no quadruple points} \} \to \Z/2\Z, $$ 
	which is invariant under regular homotopies. This is an example of a \textit{local invariant}.
	Such invariants have further been studied for arbitrary oriented surfaces in \cite{Gor_Local,Nowik_Quad, Nowik_AutoQInv,Nowik_qInv}. They are special cases of \textit{finite order} invariants, originally introduced in knot theory by Vassiliev \cite{Vass_Cohom} (cf.\ \cite{CDM_VassilievKnotInv}) and later extended to immersed surfaces in \cite{Nowik_OrderOne,Nowik_HigherOrder}. They define a large class of invariants that provide insight into the structure of the space of immersed surfaces.
	
	However, the classification of all finite order invariants in \cite{Nowik_HigherOrder} shows that such invariants cannot separate any triple-point-free immersion from \emph{both} standard embedded spheres by triple or quadruple points.
	We present an invariant that serves this purpose for triple points; it is therefore necessarily not of finite order.
	
	The paper is organized as follows. The invariant and its properties are stated in Section~\ref{sec_intro_main} whereas in Sections~\ref{sec_intro_Immn} and~\ref{sec_intro_Willmore} some applications are discussed. Section~\ref{sec_prelim} fixes notation and introduces necessary background. In Section~\ref{sec_tree} each immersed sphere without triple points is associated to a directed tree which encodes its double point structure. Section~\ref{sec_treeEH} contains the classification of modifications that a regular homotopy can induce on such trees. Finally, Sections~\ref{sec_Invariant} and~\ref{sec_image} contain the proofs of the main Theorems~\ref{thm_Invariant} and~\ref{thm_imageF}, respectively.
	
	\subsection{Main Results}\label{sec_intro_main}
	We denote by $ \Imm(\S^2,\R^3) $ the space of smooth immersions $\S^2\to\R^3$ and for $ n\geq 2 $
	\begin{equation*}
		\Imm_{<n}(\S^2,\R^3):=\{f\in \Imm(\S^2,\R^3)\ | \ \text{f has no } n \text{-tuple points}\}
	\end{equation*}
	denotes the set of immersed spheres without $ n $-tuple points.
	The invariant is first defined on generic immersions and then extended to all immersions in $ \Imm_{<3}(\S^2,\R^3) $ (see Proposition~\ref{prop_extension}). Given a generic immersion $ f\in \Imm_{<3}(\S^2,\R^3) $, we associate to it a directed tree $ G_f=(V_f,E_f) $ and a topological degree $ \delta_f:V_f\to\Z $ as described in Section~\ref{sec_tree}. For a vertex $ v\in V_f $ we denote by $ \deg^-(v) $ the \textit{indegree} of $ v $, that is, the number of edges that point into $ v $.
	\begin{defn}
		We define
		\begin{equation*}
			F:\Imm_{<3}(\S^2,\R^3)\to \Z^\Z,\qquad f\mapsto \left[\sum_{v\in \delta_f^{-1}(\{k\})} \left( 1-\deg^-(v)\right)\right]_{k\in\Z}.
		\end{equation*}
	\end{defn}
	\begin{theorem}\label{thm_Invariant}
		$ F $ is invariant under regular homotopies in $ \Imm_{<3}(\S^2,\R^3) $.
	\end{theorem}
	\begin{theorem}\label{thm_imageF}
	The image of $ F $ is given by
		\begin{align*}
			\left\{ (h_k)_{k\in\Z} \in \Z^\Z \ \middle| \ h_{2k} = 0 \text{ for all } k\in\Z \text{ and }\sum_{k\in\Z} |h_k|<\infty \text{ and } \sum_{k\in\Z} h_k = 1 \right\}.
		\end{align*}
	\end{theorem}
	Interestingly, the image is entirely realized by surfaces of revolution (see Section~\ref{sec_image}).
	
	\begin{rem}
		Recall that the space of diffeomorphisms $ \text{Diff}(\S^2) $ has two homotopy classes which contain the orientation-preserving and -reversing diffeomorphisms \cite{Smale_DiffS2}. With Lemma~\ref{lem_delta-f}~\ref{lem_delta-f_-f}, the invariant $ F $ restricts to $ \Imm_{<3}(\S^2,\R^3) / \text{Diff}(\S^2) \to \Z^{\Z}/\sim $ by identifying $ (h_k)_{k\in\Z} \sim (l_k)_{k\in\Z} $ if $ h_k = l_{\m k} $.
	\end{rem}

	\subsection{The Space of Immersed Spheres without $ n $-tuple Points}\label{sec_intro_Immn}
	
	It is unknown whether the invariant $ F $ is complete. However, Theorem~\ref{thm_imageF} shows that it already divides the space $ \Imm_{<3}(\S^2,\R^3) $ into infinitely many components. Previously, it was known---via the invariant of \cite{MB_SE}---that this space has at least two components. Table~\ref{tab_RHCNtuple} provides an overview of the number of components of the space of immersed spheres without $ n $-tuple points.
	
	\begin{table}[H]
		\centering
		\begin{tabular}{@{} lcl @{}}  
			\toprule
			Space & Number of regular homotopy classes &  \\
			\midrule
			$ \Imm(\S^2,\R^3) $   & $ 1 $         & \cite{Smale_SE} \\
			$ \Imm_{<2}(\S^2,\R^3) $ & $ 2 $   & \cite{Alexander_EmbSpheres} \\
			$ \Imm_{<3}(\S^2,\R^3) $ & $ \infty $  & Theorem~\ref{thm_imageF} \\
			$ \Imm_{<4}(\S^2,\R^3) $ & $ \geq 2 $   & \cite{MB_SE} \\
			$ \Imm_{<n}(\S^2,\R^3) $, $ n\geq 5 $ &  1   & \\
			\bottomrule
		\end{tabular}
	\caption{The number of regular homotopy classes of the space $ \Imm_{<n}(\S^2,\R^3) $ of immersed $ 2 $-spheres without $ n $-tuple points.}
	\label{tab_RHCNtuple}
	\end{table}
	
	For $ n\geq 5 $, the space $ \Imm_{<n}(\S^2,\R^3) $ consists of one regular homotopy class due to the fact that $ n $-tuple points for $ n\geq 5 $ have codimension greater than $ 1 $: every regular homotopy can be perturbed to not have $ n $-tuple points for $ n\geq 5 $ (cf.\ Section~\ref{sec_prelim}).
	
	Table~\ref{tab_RHCNtuple} raises a natural question.
	\begin{openq}\label{q_RHCwithoutQ}
		Does the space $ \Imm_{<4}(\S^2,\R^3) $ of immersed spheres without quadruple points have more than two regular homotopy classes?
	\end{openq}
	\com{One can check that Nowik's universal first order invariant $ f $ from \cite{Nowik_OrderOne} normalized at $ j $ evaluates to $ f(j)=0 $, $ f(e)=Q^2_0 $ and $ f(-e)= -H^2_1 $, where both $ e $ and $ j $ have outer normal.}
	
	\subsection{Applications to Willmore energy}\label{sec_intro_Willmore}
	One application of this invariant is in the study of the Willmore energy of immersed spheres \cite{MBS_Landscape}. 
	The Willmore energy is the integral of the squared mean curvature and linked to self-intersections by the Li-Yau inequality \cite{LY}: An immersion with an $ n $-tuple point has Willmore energy of at least $ 4\pi n $. The invariant helps to describe the energy landscape and understand the singular behavior of the associated gradient flow. 
	Consider a smooth immersion $ j:\S^2\to\R^3 $ obtained by rotating the curve displayed in Figure~\ref{fig_sketchJ}. Surfaces of this kind were first treated in \cite{MS_Num,Blatt} as initial surfaces for singular Willmore flows.
	\begin{figure}[H]
		\centering
		\vspace*{-0.05\textwidth}
		\includegraphics[align=c,width=0.49\textwidth]{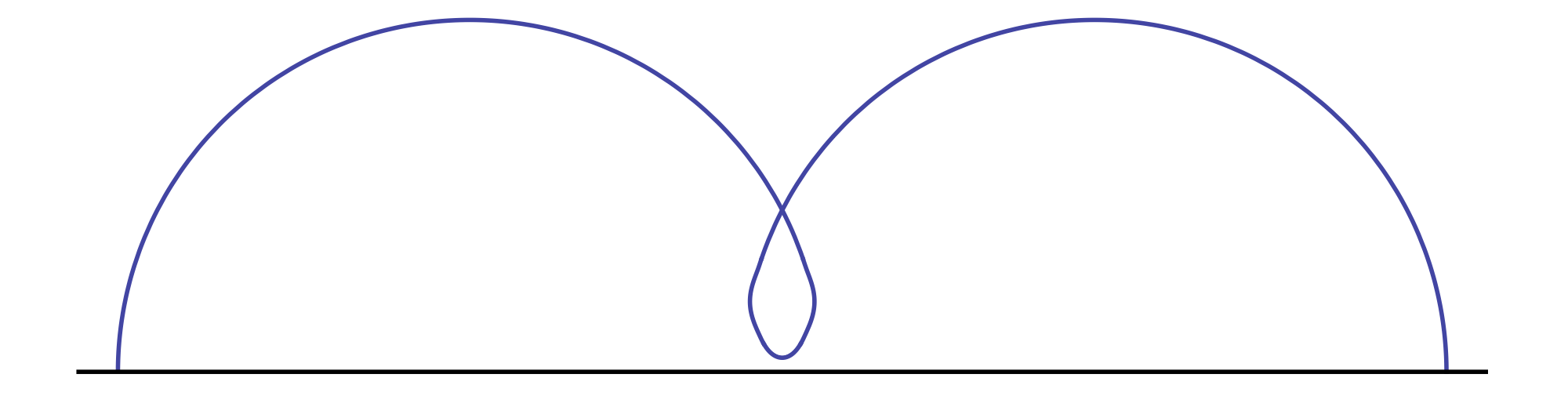}
		\includegraphics[align=c,width=0.49\textwidth]{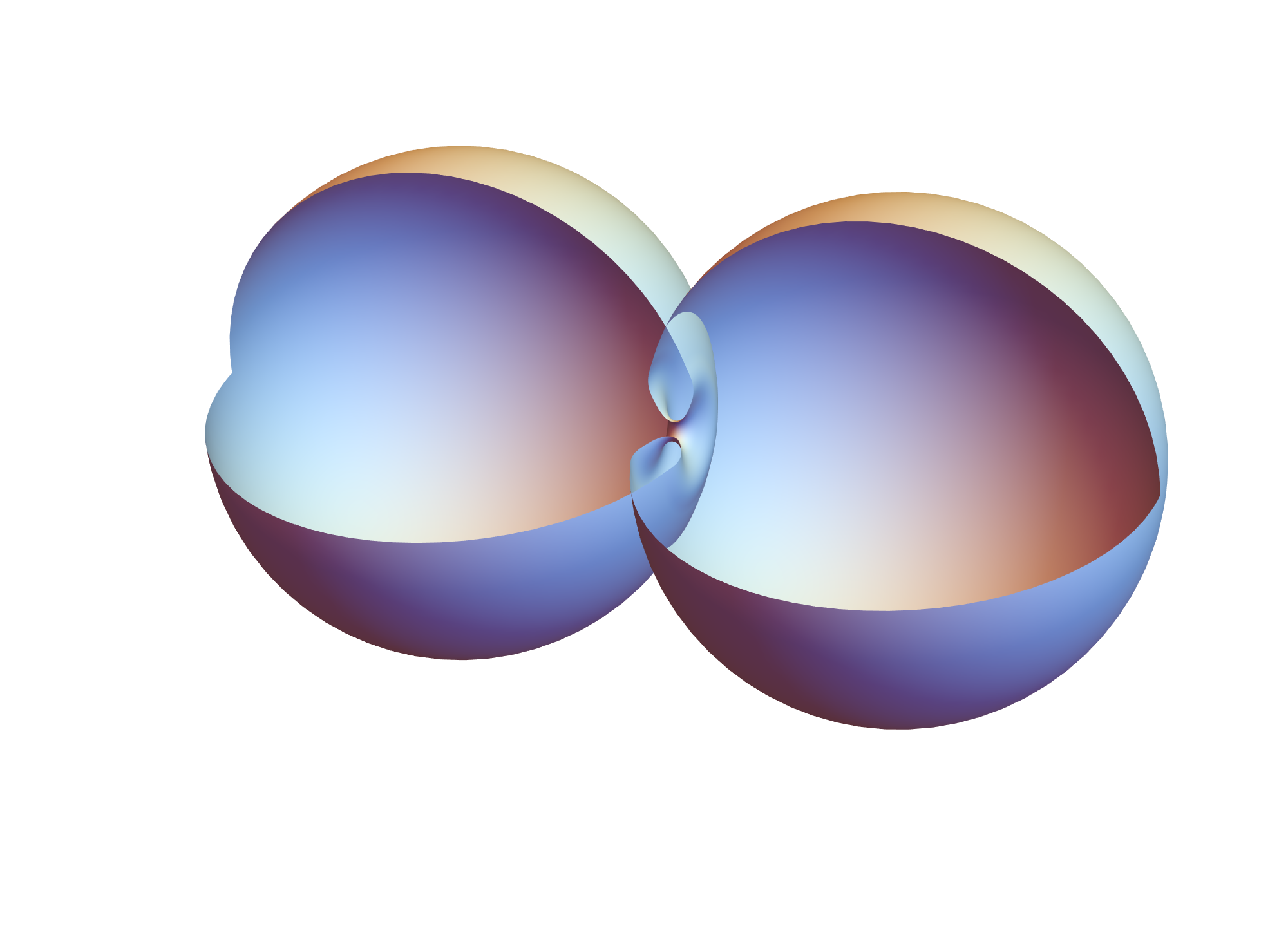}
		\vspace*{-0.05\textwidth}
		\caption{The immersion $ j $ and its generating curve.}\label{fig_sketchJ}
	\end{figure}
	\begin{theorem}\label{thm_triplePointsE3}
		Every regular homotopy joining $ j $ to an embedded sphere has triple points.
	\end{theorem}
	\begin{proof}
		This follows from Theorem~\ref{thm_Invariant} and Example~\ref{ex_Fej}.
	\end{proof}
	Theorem~\ref{thm_triplePointsE3} provides a topological explanation for the appearance of singularities in the Willmore flow. It also yields a lower bound on connected components of sublevel sets of the Willmore energy. See \cite{MBS_Landscape} for further details. 
	
	The surface $ j $ may serve as a candidate for approaching Question~\ref{q_RHCwithoutQ}.
	\begin{openq}\label{q_quadruplePoints}
		Does every regular homotopy joining $ j $ to an embedded sphere have at least one quadruple point?
	\end{openq}
	\section{Preliminaries}\label{sec_prelim}
	Consider the map $ \chi:\R^3\to\R^3 $ with $ \chi(x)=-x $. For maps $ f:\Sigma\to\R^3 $ and subsets $ A\subseteq \R^3 $ we write $ -f:=\chi\circ f $ and $ -A:=\chi(A) $.
	We consider the standard embedding $ e:\S^2\to\R^3 $ and fix by convention the oriented normal of $ e $ to be outward facing.
	For an interval $ I\subseteq \R $ and a homotopy $ H: X\times I\to Y$ we write $ H_t:=H(\cdot,t) $.
	
	We endow $ \Imm(\S^2,\R^3) $ with the $C^\infty$-topology. In view of the Whitney Approximation Theorem (cf.\ \cite{Hirsch_DT}), the following also applies to the space of $ C^k $-immersions, $ k\geq 1 $, such as $ C^2 $-immersions (as in \cite{MB_SE}) or $ C^1 $-immersions (as in \cite{Nowik_OrderOne}).
	
	The space of immersions $ \Imm(\Sigma,\R^3) $ of a closed surface $ \Sigma $ admits a stratification as described in \cite[Sect.\ 1.1]{Gor_Local} and \cite[Sect.\ 2]{Nowik_OrderOne}.
	We denote the union of the strata of codimension $ 0 $ by $ I_0\subseteq \Imm(\S^2,\R^3) $ and the union of the strata of codimension $ 1 $ by $ I_1\subseteq\Imm(\S^2,\R^3) $\com{Codimension refers to the multigerm codimension, w.r.t.\ the dimension of the space of variations of the germ $ f:\S^2,\{p_1,\dots,p_r\}\to \R^3,\{0\} $ with $ r\in\{2,3,4\} $ see \cite{HobbsKirk}}. 
	The subset $ I_0\subseteq \Imm(\S^2,\R^3) $ is open and dense and the image of an immersion $ i\in I_0 $ is locally equal to the union of one, two or three coordinate planes, up to ambient isotopy. We call any immersion $ i\in I_0 $ \textit{generic}.
	
	The types of strata in $ I_1 $ are classified up to ambient isotopy \cite{HobbsKirk} (see also \cite{Mond_Germs,Gor_Local,WAO_Multigerms}). They are denoted by $ E, H, T, Q\subseteq I_1 $, corresponding to elliptic and hyperbolic self-tangencies, and to triple and quadruple points. Any immersion in $ I_1 $ is generic outside a ball in $ \R^3 $. Inside this ball it is, up to ambient isotopy, of the following form (with $ \lambda=0 $).
	
	\begin{tabular}{rllll}
	$ E: $ & $ z=x^2+y^2+\lambda $,\quad& $ z=0 $,\quad & &\\
	$ H: $ & $ z=x^2-y^2+\lambda $,\quad & $ z=0 $,\quad & &\\
	$ T: $ & $ z=x^2+y\phantom{^2}+\lambda $,\quad & $ z=0 $,\quad & $ y=0 $,\quad &\\
	$ Q: $ & $ z=x\phantom{^2}+y\phantom{^2}+\lambda $,\quad & $ z=0 $,\quad & $ y=0 $,\quad & $ x=0 $. \\
	\end{tabular}\linebreak
	A transverse crossing of a codimension one stratum is represented by varying $ \lambda\in\R $ in the above normal forms (see Figure~\ref{fig_EHTQ}). 
	Note that immersions in $ Q $ and $ T $ have only transverse self-intersections but are not considered generic.
		
	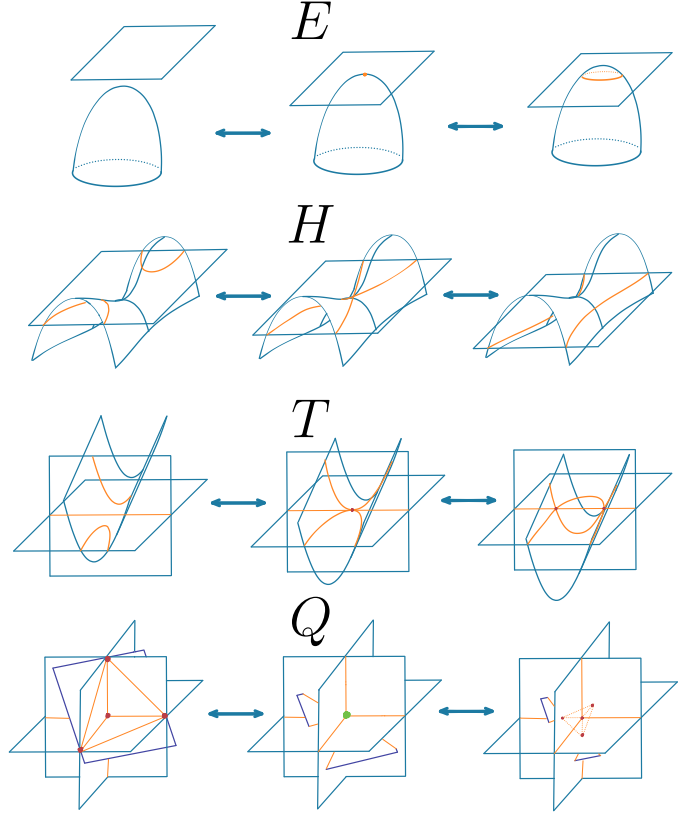
\begin{figure}[H]
		\centering
		\begin{tikzpicture}
			\def\dd{0.03*\textwidth}
			\def\dy{0.04*\textwidth}
			
			\node[anchor=south] at (0,0) {
				\includegraphics[width=0.6\textwidth]{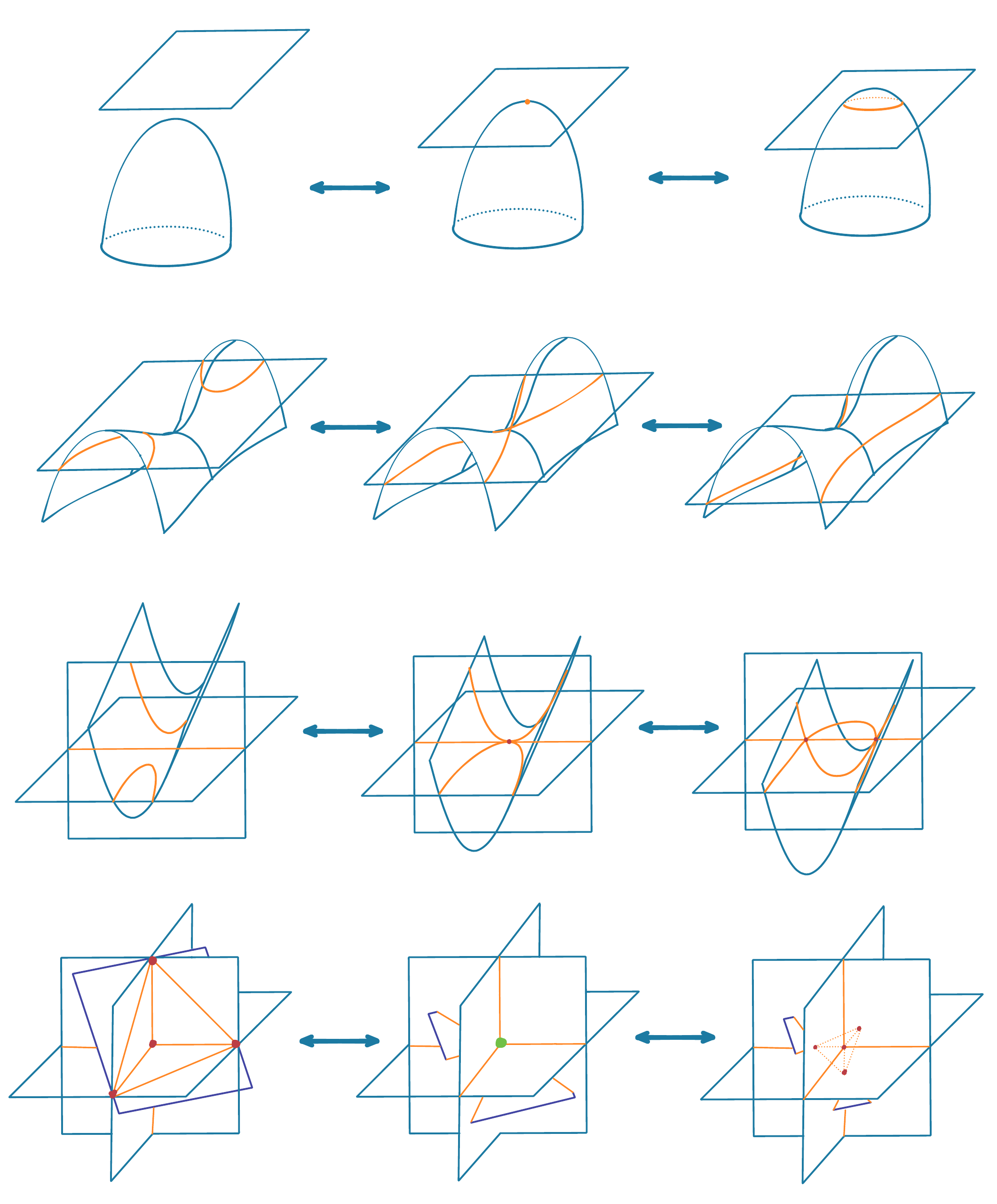}};
			\node at (-\dd,18*\dy) {\huge $ E $};
			\node at (-\dd,13.5*\dy) {\huge $ H $};
			\node at (-\dd,9.2*\dy) {\huge $ T $};
			\node at (-\dd,4.7*\dy) {\huge $ Q $};
		\end{tikzpicture}
		\caption{Transverse crossings of the codimension $ 1 $ strata of $ \Imm(\Sigma,\R^3) $ with indication of double point curves, triple points and the quadruple point.}\label{fig_EHTQ}
	\end{figure}
	
	Any regular homotopy $ H $ can be slightly perturbed such that $ H_t $ is generic for all but finitely many times $ t_1,\dots,t_k $ at which it is contained in one of the codimension one strata $ E $, $ H $, $ T $ or $ Q $. We call such a regular homotopy \textit{generic}.
	
	\section{The Double Point Tree of a Generic Immersed Sphere without Triple Points}\label{sec_tree}
	Throughout this section let $ f\in\Imm_{<3}(\S^2,\R^3) $ be generic.
	Consider the double point set
	\begin{equation*}
		X_2(f):=\left\{ x\in\R^3 \ \mid\ f^{-1}(\{x\}) \textit{ has exactly two elements} \right\}.
	\end{equation*}
	Locally around a point $ x\in X_2(f) $, the image of $ f $ is, up to ambient isotopy, the intersection of two coordinate planes. Together with the assumption that $ f $ has no triple points, this implies the following.
	The preimage of each connected component of $ X_2(f) $ consists of two circles.
	Let $ \alpha_i,\beta_i $ with $ i=1,\dots,m $ be pairs of parametrizations $ \alpha_i,\beta_i:\S^1\to\S^2 $ of these circles such that $ f\circ \alpha_i = f\circ \beta_i$. We call $ \alpha_i,\beta_i $ and also their images the \textit{double point curves} of $ f $, and we refer to $ \beta_i $ as the \textit{conjugate} of $ \alpha_i $.
	Self-conjugate curves do not exist in the absence of triple points (see the lemma in \cite[p.\ 411]{Banch_TPSurgery}).
	This yields a finite collection of pairwise disjoint double point curves on $ \S^2 $.
	
	We define a graph $ G_f=(V_f,E_f) $ where the vertices $ V_f $ are the connected components of $ \S^2 $ with the double point curves removed. The set of edges $ E_f $ contains all pairs of components whose boundaries have nonempty intersection. Since the double point curves are disjoint, a common boundary between two components always consists of a single double point curve. Therefore, we can identify $ E_f $ with the set of double point curves. 
	\begin{defn}
		A map $ P:X\to X $ on a set $ X $ is called a \textit{pairing} if $ P \circ P=\id $ and $ P(x)\neq x $ for all $ x\in X $.\com{short: fixed-point-free involution}
	\end{defn}
	We observe that $ G_f $ is a tree with a pairing $ P_f:E_f\to E_f $ defined by $ P_f(\alpha_i):=\beta_i $. 
	
	\begin{rem}
	Such tree representations are the higher dimensional analogue for Gauss codes of self-intersecting curves \cite{Rosenstiehl_GaussCodes}. 
	A natural question is the classification of \textit{realizable} trees---those that represent the self-intersections of an actual immersion. There is a classification of realizable trees with a pairing \cite{Lippner,Kalmar_DoublePointsKnottedSpheres}. For the general case of generic immersed spheres (with triple points) the classification of realizable graphs remains open (see \cite{Nowik_DissectingS2,Borbely}).
	\end{rem}
	
	We want to endow the tree $ G_f $ with additional information in two steps. First, we turn it into a \textit{directed} tree, using the pairing:
	A pair of double point curves $ (\alpha_i,\beta_i) $ divides $ \S^2 $ into two disk components and one annulus component. Hence, each edge connects a component contained in one of the disks to a component contained in the annulus. We orient the edge toward the disk component. In the tree, hence, each edge is directed away from its conjugate edge along the unique path connecting the pair.
	
	Second, we introduce a function $ \delta_f:V_f\to\Z $ which we call the \textit{(local) topological degree} of $ f $. Let us first associate to each connected component $ C $ of $ \R^3\setminus f(\S^2) $ its topological degree. Pick any point $ p\in C $ and consider the radial projection of $ f $ onto $ \partial_1 B(p)\simeq \S^2 $, i.e. $ f_p(x)  = \frac{f(x)-p}{|f(x)-p|} $. This is a differentiable map $ f_p:\S^2\to\S^2 $ and we set
	$$ \deg_f C:= \deg f_p, $$
	where $ \deg $ denotes the topological degree (cf.\ \cite{Bred}). This is well-defined since $ f_p $ and $ f_q $ are homotopic for all $ p,q\in C $, i.e.\ take a path $ c:[0,1]\to C $ from $ p $ to $ q $ and the homotopy $ t\mapsto f_{c(t)} $.
	For $ v\in V_f $ the image set $ f(v) $ has exactly two adjacent connected components $ C_1,C_2 $ of $ \R^3\setminus f(\S^2) $ and we set 
	$$ \delta_f(v):= \deg_f C_1 + \deg_f C_2. $$
	\com{The degree $ \delta_f $ is not to be confused with the topological degree of the normal of $ f $ which is always $ 1 $ since all immersed spheres in $ \R^3 $ are regularly homotopic, i.e.\ all normal maps are homotopic.}
	
	\begin{defn}\label{def_tree}
		A \textit{double point tree} is a triple $ (G,P,\delta) $ of a directed tree $ G=(V,E) $, a pairing $ P:E\to E $, and a map $ \delta:V\to\Z $.
		The \textit{double point tree} of a generic immersion $ f\in \Imm_{<3}(\S^2,\R^3) $ is the double point tree $ (G_f,P_f,\delta_f) $.
	\end{defn}
	
	\begin{lem}\label{lem_delta-f}
		For generic $ f\in \Imm_{<3}(\S^2,\R^3) $ we have for all $ v\in V_f $ and $ (v,w)\in E_f $
		\begin{enumerate}[label=\roman*)]
			\item $ \delta_f(v) $ is odd,\label{lem_delta-f_odd}
			\item $ \delta_f(w) = \delta_f(v)\pm 2 $,\label{lem_delta-f_pm2}
			\item $ \delta_{-f} = -\delta_f $.\label{lem_delta-f_-f}
		\end{enumerate}
	\end{lem}
	\begin{proof}
		\ref{lem_delta-f_odd} The functions $ f_p:\S^2\to\S^2 $ are differentiable, so the degree can be computed by counting the preimages of a regular value with orientation (cf.\ \cite{Bred}). If $ p $ passes through one sheet of the image of $ f $, then by this definition the degree of $ f_p $ changes by $ \pm 1 $. It follows that $ \delta_f $ has odd values. \\
		\ref{lem_delta-f_pm2} Now, take some point $ q $ on $ f(v) $ and let $ \nu(q) $ denote the normal of $ f $ in $ q $. For small $ \eps>0 $ we have $ \delta_f(v) = \deg f_{q+\eps\nu(q)} + \deg f_{q-\eps\nu(q)} $. If we move $ q $ on the image of $ f $ into $ f(w) $, it passes through a sheet of $ f $ and by the same argument as above, the degrees of  $ f_{q+\eps\nu(q)} $ and $ f_{q-\eps\nu(q)} $ both increase by $ 1 $ or both decrease by $ 1 $.\\
		\ref{lem_delta-f_-f} We identify the images of $ f $ and $ -f $ by the map $ x\mapsto -x $. As $ -f $ has opposite orientation to $ f $, the degrees of the connected components of $ \R^3\setminus (-f)(\S^2) $ and $ \R^3 \setminus f(\S^2) $ satisfy $ \deg_{-f} -C = - \deg_f C $.
	\end{proof}
	
	\begin{ex}\label{ex_eTree}
		For the standard embedding $ e $ we have $ V_e = \{v\},\ E_e = \emptyset $, $ \delta_e(v)=1 $ and for $ -e $ we have $ V_{\m e} = \{v\},\ E_{\m e} = \emptyset $, $ \delta_{\m e}(v)=-1 $. 
	\end{ex}
	\begin{ex}\label{ex_jTree}
		For $ j $ as depicted in Figure~\ref{fig_sketchJ} we have (see Figure~\ref{fig_jTree})
		\begin{align*}
			V_j &= \{v,w_1,w_2\},\\
			E_j &= \{(v,w_1),(v,w_2)\},\\
			P_j((v,w_1)) &= (v,w_2),\\
			\delta_j(v)&=3,\quad \delta_j(w_1)=\delta_j(w_2)=1.
		\end{align*}
		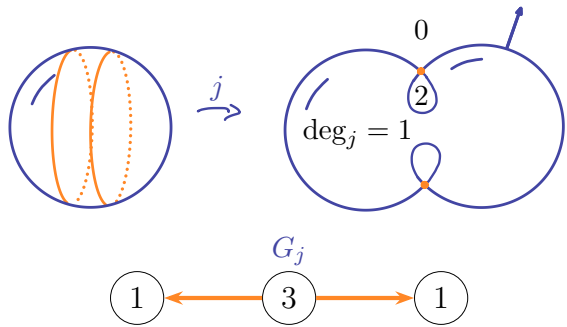
\begin{figure}[H]
			\centering
			\def\dx{2}
			\begin{tikzpicture}
				\def\dd{0.1*\textwidth}
				
				\node[anchor=south] at (0,0) {	\includegraphics[width=0.5\textwidth]{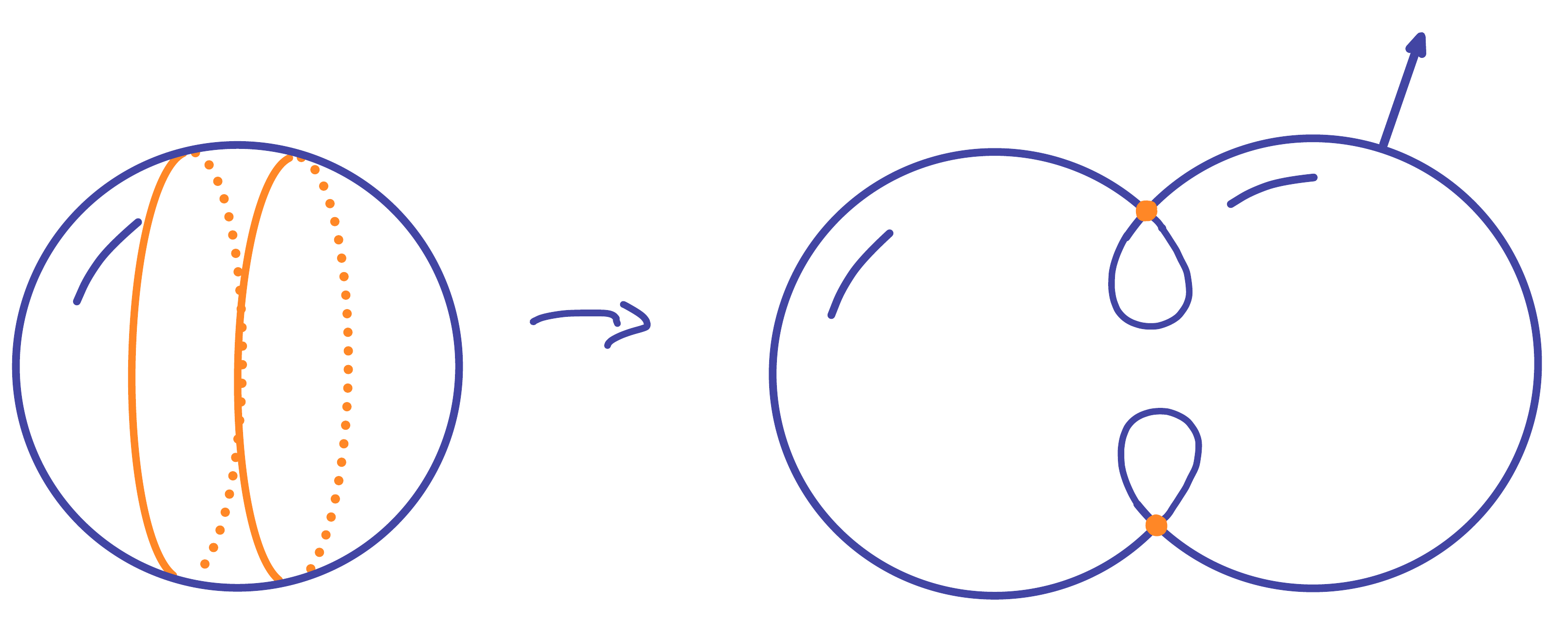}};
				\node[anchor=south] at (-0.63*\dd,1.05*\dd) {\textcolor{myblue}{$ j $}};
				\node[anchor=center] at (0.6*\dd,0.85*\dd) {$ \deg_j = 1 $};
				\node[anchor=south] at (1.16*\dd,1.02*\dd) {$ 2 $};
				\node[anchor=south] at (1.16*\dd,1.6*\dd) {$ 0 $};
				
				
				\node[anchor=south] at (0*\dd,-0.4*\dd) {\textcolor{myblue}{$ G_j $}};
				
				\begin{scope}[shift={(0,-0.6*\dd)}]
					\node[hexnode] (v) at \hexcoord{\dx}{0}{0} {\large $3$};
					\node[hexnode] (w1) at \hexcoord{\dx}{-1}{0} {\large $1$};
					\node[hexnode] (w2) at \hexcoord{\dx}{1}{0} {\large $1$};
					\draw[directed, ultra thick, myorange,>= {Stealth[length=7pt,width=5pt]}] (v) -- (w1);
					\draw[directed, ultra thick, myorange,>= {Stealth[length=7pt,width=5pt]}] (v) -- (w2);
				\end{scope}
				
			\end{tikzpicture}
			\def\dx{\dxdefault}
			\caption{An illustration of the tree $ G_j $ for $ j $ with the values of $ \delta_j $ indicated on the vertices. Matching colors indicate pairs of double point curves.}\label{fig_jTree}
		\end{figure}
	\end{ex}

	\section{Tree Modifications under Regular Homotopies without Triple Points}\label{sec_treeEH}
	We translate the singularities of type E and H to the double point tree to give a classification of the modifications $ G_{H_t} $ undergoes along a regular Homotopy $ H $ without triple points.
	The proofs are mostly combinatorial and elementary except for Lemma~\ref{lem_noSelfTypeH} where we use that every immersed real projective plane in $ \R^3 $ necessarily has triple points \cite{Banch_TPProjections,Banch_TPSurgery}.
	
	During a transverse crossing of an $ E $ singularity a pair of points appears on $ \S^2 $ and grows into a new pair of double point curves, or a pair of double point curves shrinks to points and vanishes. For $ H $ singularities the modification on the double point tree is more involved. Both are covered in Proposition~\ref{prop_classTreeMod} after some preparations to describe the $ H $ singularities.
	
	By the normal forms described in Section~\ref{sec_prelim}, we know the local structure of the double point curves on $ \S^2 $ involved in an $ H $ singularity. There exist two balls on $ \S^2 $ such that, up to diffeomorphism, the double point curves intersecting these balls consist of two transversely intersecting arcs in each ball (see Figure~\ref{fig_HDPCLocalCrossing}). Depending on the orientations of the two surface sheets at the $ H $ singularity, there are two possible cases. We denote these by $ H^1 $ and $ H^2 $, following the notation of \cite{Nowik_OrderOne} (see Figure~\ref{fig_HDPClocal}).
	\begin{figure}[H]
		\centering
		\includegraphics[width=0.6\textwidth]{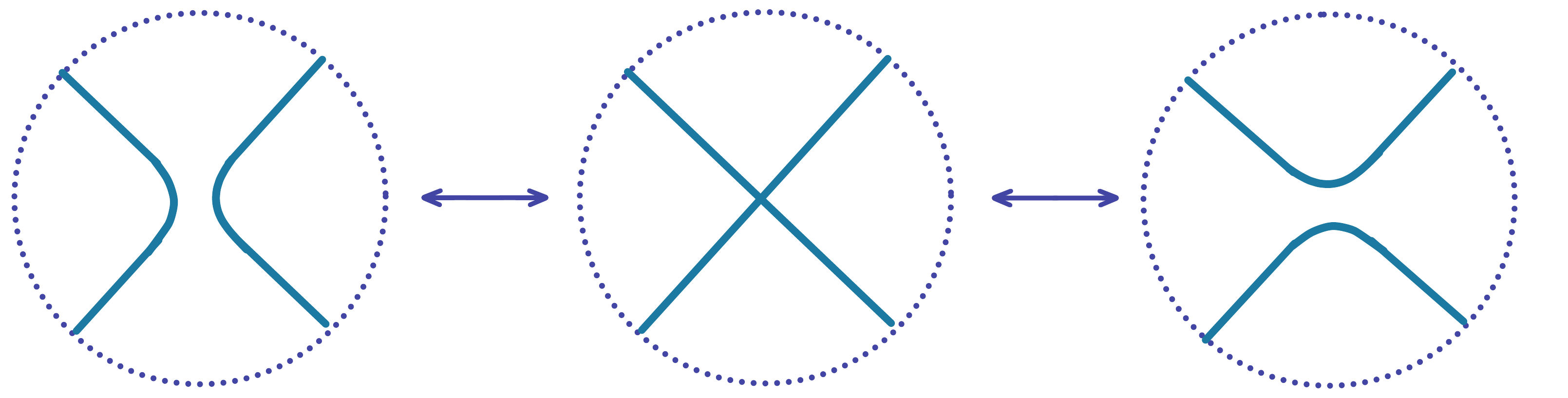}
		\caption{The preimage of one sheet involved in a transverse crossing of an $ H $ singularity.}\label{fig_HDPCLocalCrossing}
	\end{figure}
	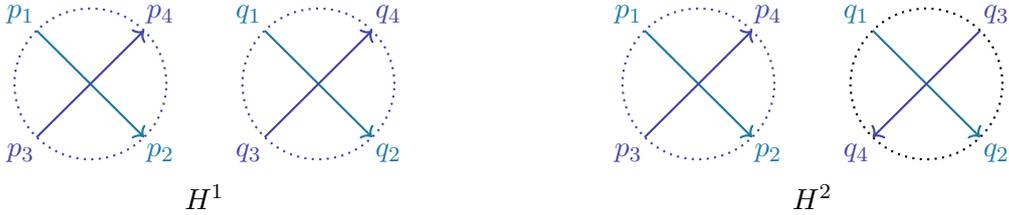
\begin{figure}[H]
		\centering
		\hspace{0.02\textwidth}
		\begin{tikzpicture}
			\def\prad{1.3}
			\newcommand{\drawDiagram}[4]{
				\begin{scope}[shift={#1}]
					\draw[dotted, thick,myblue] (0,0) circle(1);
					
					\node at (135:\prad) {\textcolor{mypastel}{$#2_1$}};
					\node at (315:\prad) {\textcolor{mypastel}{$#2_2$}};
					\node at (225:\prad) {\textcolor{myblue}{$#2_3$}};
					\node at (45:\prad)  {\textcolor{myblue}{$#2_4$}};
					
					\draw[->, thick, mypastel] (135:1) -- (315:1);
					\draw[->, thick, myblue] (225:1) -- (45:1);
				\end{scope}
			}
			
			\drawDiagram{(0,0)}{p}{1}{4}
			\drawDiagram{(3,0)}{q}{1}{4}
			\drawDiagram{(8,0)}{p}{1}{4}
			
			\node at (1.5,-1.5) {$ H^1 $};
			\node at (9.5,-1.5) {$ H^2 $};
			
			\begin{scope}[shift={(11,0)}]
				\draw[dotted, thick] (0,0) circle(1);
				\node at (135:\prad) {\textcolor{mypastel}{$q_1$}};
				\node at (315:\prad) {\textcolor{mypastel}{$q_2$}};
				\node at (225:\prad) {\textcolor{myblue}{$q_4$}};
				\node at (45:\prad)  {\textcolor{myblue}{$q_3$}};
				
				\draw[->, thick, mypastel] (135:1) -- (315:1);
				\draw[->, thick, myblue] (45:1) -- (225:1);
			\end{scope}
			
		\end{tikzpicture}
		\hspace{0.02\textwidth}
		\caption{The preimages of an $ H^1 $ and $ H^2 $ singularity on $ \S^2 $.}\label{fig_HDPClocal}
	\end{figure}
	We classify how these intersecting arcs might close up to closed double point curves.
	\begin{lem}\label{lem_Hconnectivity}
		Let $ f\in H \cap\, \Imm_{<3}(\S^2,\R^3) $ be a generic immersion and 
		$$ A:=\{p_1,p_2,p_3,p_4,q_1,q_2,q_3,q_4\}\subseteq \S^2 $$ 
		denote the endpoints of the arcs above such that $ f(p_i)=f(q_i) $ and $ p_1,p_2 $ lie on the same arc as well as $ p_3,p_4 $. 
		Then, these points are connected by double point curves described by some relation $ R\subseteq \left\{S\subseteq A \ \mid\ \abs{S}=2 \right\} $.
		It satisfies 
		\begin{align*}
			R=\left\{\{p_1,p_i\},\ \{p_2,p_j\},\ \{q_1,q_i\},\ \{q_2,q_j\}\right\},
		\end{align*}
		where $ i,j\in\{3,4\} $ and $ i\neq j $.
	\end{lem}
\begin{figure}[h]
	\centering
	\includegraphics[width=0.9\textwidth]{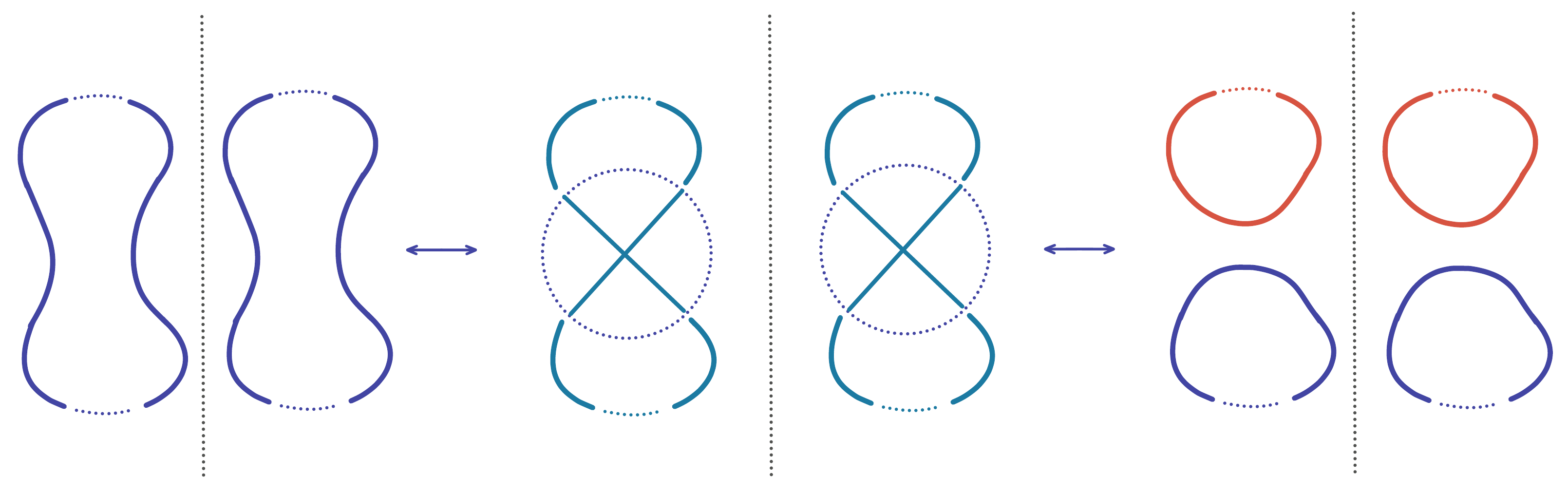}
	\caption{The connectivity of $ H $ singularities according to Lemma~\ref{lem_Hconnectivity}.}\label{fig_Hconnectivity}
\end{figure}
	\begin{proof}
		We first consider the case that there is no connection $ \{p_i,q_j\}\in R $ for any $ i,j $. 
		Then, $ \{p_1,p_2\}\notin R $, otherwise its curve would intersect the curve $ \{p_3,p_4\} $. Now, if $ \{p_1,p_3\}\in R $, then $ \{q_1,q_4\}\notin R $ and $ \{q_1,q_2\}\notin R $, otherwise resolving the $ H $ singularity yields an odd number of double point curves. Analogously, if $ \{p_1,p_4\}\in R $, then $ \{q_1,q_4\}\in R$. We obtain the above listed connections.\\
		For the remaining cases we can assume that $ \{p_1,q_j\}\in R $ for some $ j $. It follows $ j=2 $. Otherwise, resolving the $ H $ singularity yields double point curves which contain both $ p_1 $ and $ q_1 $. Such self-conjugate double-point curves cannot appear (cf.\ Section~\ref{sec_prelim}).
		For the same reason, $ p_2 $ has to be joined to $ q_1 $, and not $ p_3 $ or $ p_4 $.
		Now consider the case of the $ H^2 $ singularity (see Figure~\ref{fig_HDPClocal}). Then, there are no valid connections of the remaining points $ p_3,p_4,q_3,q_4 $ without intersecting each other. However, for the case of $ H^1 $, we obtain the possible configuration (see also Figure~\ref{fig_HselfTypeConn})
		\begin{align}\label{eq_HselfType}
			\left\{\{p_1,q_2\},\{p_2,q_1\},\{p_3,q_4\},\{p_4,q_3\}\right\}.
		\end{align}
		The following Lemma shows that this cannot be the case and the claim follows.
	\end{proof}
	\begin{figure}[H]
		\centering
		\includegraphics[width=0.4\textwidth]{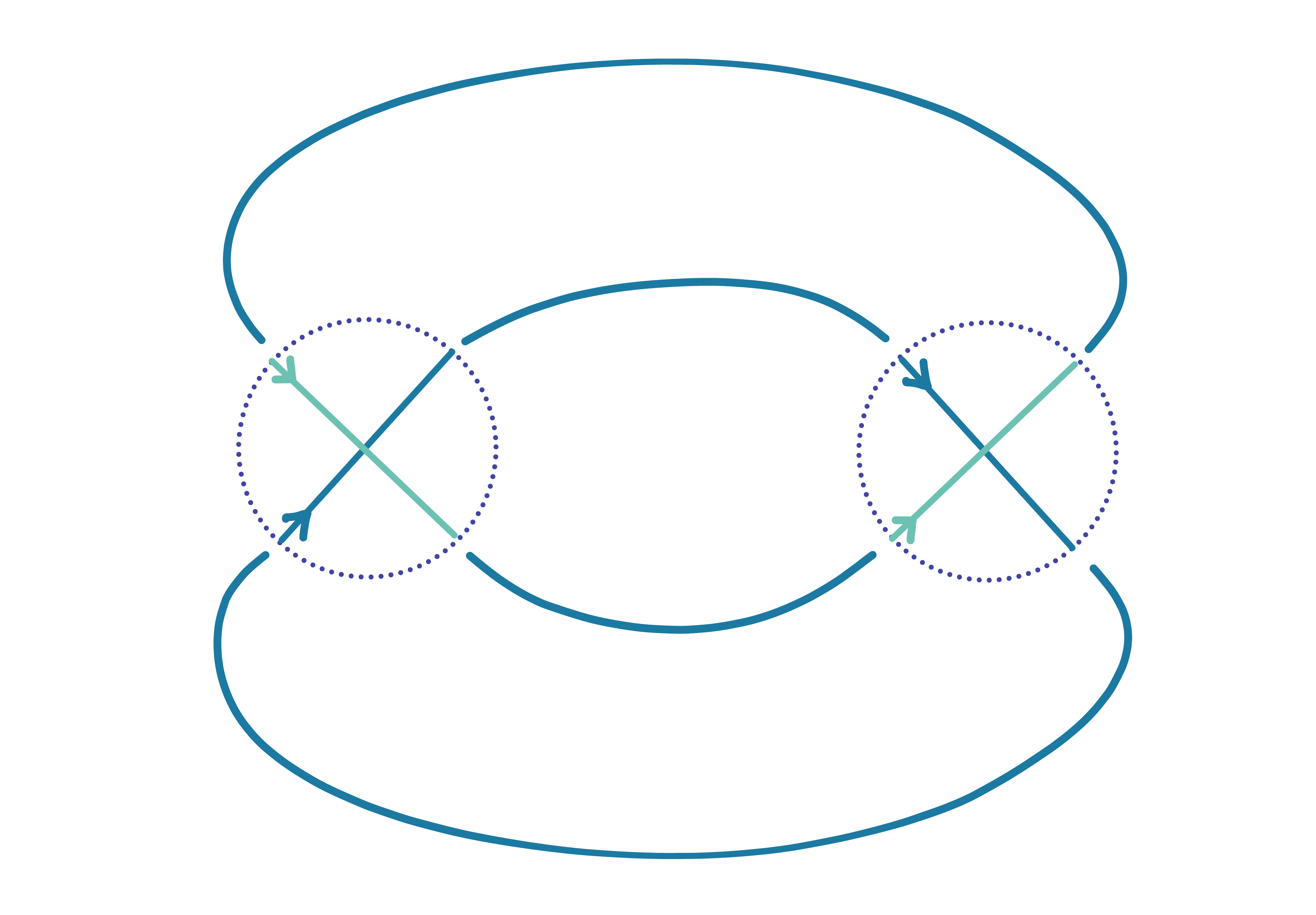}
		\caption{The connections of \eqref{eq_HselfType}.}\label{fig_HselfTypeConn}
	\end{figure}
	\begin{lem}\label{lem_noSelfTypeH}
		A generic immersion $ f\in H\, \cap\, \Imm_{<3}(\S^2,\R^3) $ cannot have an $ H $ singularity of the type described in \eqref{eq_HselfType}.
	\end{lem}
	\begin{proof}
		Note that the double point curve $ \gamma:\S^1\to \S^2 $ that passes through the points $ p_1,p_2,q_1,q_2,p_1 $ in order has some reparametrization such that $ f(\gamma(\alpha)) = f(\gamma(-\alpha)) $ for all $ \alpha\in\S^1 $(see first sketch in Figure~\ref{fig_HselfTypeS1Nbhd}).
		By reparametrizing $ f $, we assume that the double point curves of the $ H $ singularity are two great circles $ S_1=\gamma(\S^1) $ and $ S_2 $ that intersect perpendicular and satisfy $ f(x) = f(-x) $ for $ x $ in a neighborhood of $ S_1\cup S_2 $. We then deform the immersion locally around the self-intersections without introducing triple points such that in a neighborhood of $ S_1 $ the deformed map $ \hat{f} $ parametrizes a double covered moebius band (see first modification in Figure~\ref{fig_HselfTypeFlattening}). Let $ D $ denote one of the closed hemispheres bounded by $ S_1 $. Observe that $ \hat f|_{D} $ parametrizes an immersion of $ \RP^2 $ into $ \R^3 $. By \cite{Banch_TPSurgery} such an immersion has an odd number of triple points, a contradiction to the assumption $ f\in \Imm_{<3}(\S^2,\R^3) $. See Figures~\ref{fig_HselfTypeFlattening} and~\ref{fig_HselfTypeS1Nbhd} for an illustration of this argument.
	\end{proof}
	\begin{figure}[H]
		\centering
		\includegraphics[width=0.85\textwidth]{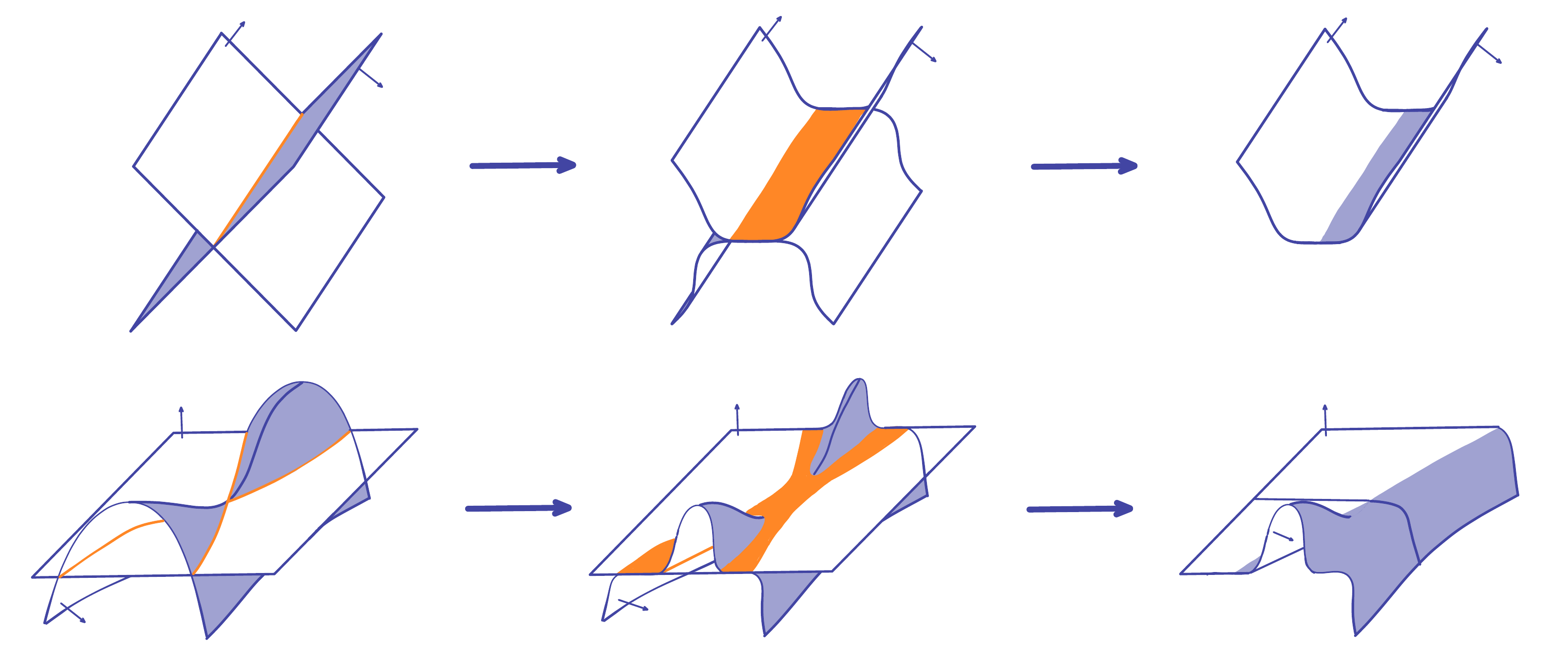}
		\caption{Local neighborhoods of points in $ f(S_1) $, $ \hat f(S_1) $ and $ \hat f|_D(S_1) $ in $ \R^3 $.}\label{fig_HselfTypeFlattening}
	\end{figure}
	\begin{figure}[H]
		\centering
		\includegraphics[width=0.95\textwidth]{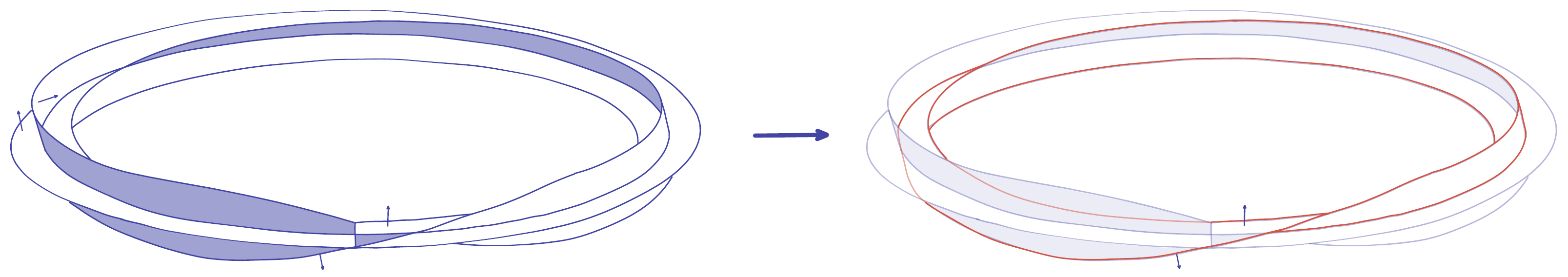}
		\caption{Neighborhoods of $ f(S_1) $ and $\hat f(S_1)|_D$ in $ \R^3 $.}\label{fig_HselfTypeS1Nbhd}
	\end{figure}
	We now carry over the classifications of double point curves during a crossing of an H singularity to the double point tree. First, we give a definition of modifications on double point trees that turn out to be essential to describe H singularities.
	\begin{defn}
		Let $ (G=(V,E),P,\delta) $ be a double point tree.
		\begin{enumerate}[label=\roman*)]
			\item A \textit{(parallel)} $ H_p $\textit{-split} on an edge $ (v,w)\in E $ yields a new double point tree with incomplete pairing. It has an added vertex $ w' $ with $ \delta(w')=\delta(w) $ and an added edge $ (v,w') $. Any edge previously attached to $ w $ may (or may not) be attached to $ w' $ instead, with its orientation and pairing unchanged.
			\item A \textit{(sequential)} $ H_s $\textit{-split} on an edge $ (v,w)\in E $ yields a new double point tree with incomplete pairing. It has an added vertex $ z $ with $ \delta(z)=\delta(v) $ and an added edge $ (w,z) $. Any edge previously attached to $ v $ may (or may not) be attached to $ z $ instead, with its orientation and pairing unchanged. If an edge that lies on the unique path from $ (v,w) $ to $ P((v,w)) $ reattached from $ v $ to $ z $ in this way, then replace $ (v,w) $ and $ (w,z) $ by $ (w,v) $ and $ (z,w) $, respectively.
		\end{enumerate}
\begin{figure}[H]
	\centering
	\def\dx{0.9}
	\begin{tabular}{l r}
		\multirow{2}{*}{
			\begin{tikzpicture}
				
				\node[hexnode] (v1) at \hexcoord{\dx}{0}{0} {$ \sigma $};
				\node[hexnode] (v2) at \hexcoord{\dx}{-1}{2} {$\tilde{\sigma}$};
				\node[hexnode,draw=none,fill=none] (w) at \hexcoord{\dx}{-1}{0} {};
				\node[trianglenode] (t-) at \hexcoord{\dx}{0.5}{0.5} {};
				\node[trianglenode] (t+) at \hexcoord{\dx}{1}{-0.5} {};
				\node[pentanode] (c-) at \hexcoord{\dx}{-0.5}{2.5} {};
				\node[pentanode] (c+) at \hexcoord{\dx}{0}{1.5} {};
				\node[squarenode] (s-) at \hexcoord{\dx}{-2}{2.5} {};
				\node[squarenode] (s+) at \hexcoord{\dx}{-1.5}{1.5} {};
				
				\draw[undirected] (v1) -- (w);
				\draw[directed, myblue] (v1) -- (v2);
				\draw[directed, gray, very thin] (t+) -- (v1);
				\draw[directed, gray, very thin] (v1) -- (t-);
				\draw[directed, gray, very thin] (s+) -- (v2);
				\draw[directed, gray, very thin] (v2) -- (s-);
				\draw[directed, gray, very thin] (c+) -- (v2);
				\draw[directed, gray, very thin] (v2) -- (c-);
				
				\node[hexnode, draw=none, fill=none] (A1) at (1.5*\dx,1) {};
				\node[hexnode, draw=none, fill=none] (A2) at (2.5*\dx,1) {};
				\draw[bidirected] (A1) -- (A2);
				
				\begin{scope}[shift={(4.75*\dx,0)}]
					\node[hexnode] (v1) at \hexcoord{\dx}{0}{0} {$ \sigma $};
					\node[hexnode] (v2) at \hexcoord{\dx}{-2}{2} {$\tilde{\sigma}$};
					\node[hexnode] (v2p) at \hexcoord{\dx}{0}{2} {$\tilde{\sigma}$};
					\node[hexnode,draw=none,fill=none] (w) at \hexcoord{\dx}{-1}{0} {};
					\node[trianglenode] (t-) at \hexcoord{\dx}{0.5}{0.5} {};
					\node[trianglenode] (t+) at \hexcoord{\dx}{1}{-0.5} {};
					\node[pentanode] (c-) at \hexcoord{\dx}{0.5}{2.5} {};
					\node[pentanode] (c+) at \hexcoord{\dx}{1}{1.5} {};
					\node[squarenode] (s-) at \hexcoord{\dx}{-3}{2.5} {};
					\node[squarenode] (s+) at \hexcoord{\dx}{-2.5}{1.5} {};
					
					\draw[undirected] (w) -- (v1);
					\draw[directed, myblue] (v1) -- (v2);
					\draw[directed, myred] (v1) -- (v2p);
					\draw[directed, gray, very thin] (t+) -- (v1);
					\draw[directed, gray, very thin] (v1) -- (t-);
					\draw[directed, gray, very thin] (s+) -- (v2);
					\draw[directed, gray, very thin] (v2) -- (s-);
					\draw[directed, gray, very thin] (c+) -- (v2p);
					\draw[directed, gray, very thin] (v2p) -- (c-);
				\end{scope}
		\end{tikzpicture}} &
		\begin{tikzpicture}
			
			\node[hexnode] (v1) at \hexcoord{\dx}{0}{0} {$ \sigma $};
			\node[hexnode] (v2) at \hexcoord{\dx}{-1}{2} {$\tilde{\sigma}$};
			\node[hexnode,draw=none,fill=none] (w) at \hexcoord{\dx}{-1}{0} {};
			\node[trianglenode] (t-) at \hexcoord{\dx}{0.5}{0.5} {};
			\node[trianglenode] (t+) at \hexcoord{\dx}{1}{-0.5} {};
			\node[pentanode] (c-) at \hexcoord{\dx}{-0.5}{2.5} {};
			\node[pentanode] (c+) at \hexcoord{\dx}{0}{1.5} {};
			\node[squarenode] (s-) at \hexcoord{\dx}{-1}{0.5} {};
			\node[squarenode] (s+) at \hexcoord{\dx}{-0.5}{-0.5} {};
			
			\draw[undirected] (v1) -- (w);
			\draw[directed, myblue] (v1) -- (v2);
			\draw[directed, gray, very thin] (t+) -- (v1);
			\draw[directed, gray, very thin] (v1) -- (t-);
			\draw[directed, gray, very thin] (s+) -- (v1);
			\draw[directed, gray, very thin] (v1) -- (s-);
			\draw[directed, gray, very thin] (c+) -- (v2);
			\draw[directed, gray, very thin] (v2) -- (c-);
			
			\node[hexnode, draw=none, fill=none] (A1) at (1.5*\dx,1) {};
			\node[hexnode, draw=none, fill=none] (A2) at (2.5*\dx,1) {};
			\draw[bidirected] (A1) -- (A2);
			
			\begin{scope}[shift={(4*\dx,0)}]
				\node[hexnode] (v1) at \hexcoord{\dx}{0}{0} {$ \sigma $};
				\node[hexnode] (v2) at \hexcoord{\dx}{-1}{2} {$\tilde{\sigma}$};
				\node[hexnode] (v2p) at \hexcoord{\dx}{-2}{4} {$\sigma$};
				\node[hexnode,draw=none,fill=none] (w) at \hexcoord{\dx}{-1}{0} {};
				\node[trianglenode] (t-) at \hexcoord{\dx}{0.5}{0.5} {};
				\node[trianglenode] (t+) at \hexcoord{\dx}{1}{-0.5} {};
				\node[pentanode] (c-) at \hexcoord{\dx}{-0.5}{2.5} {};
				\node[pentanode] (c+) at \hexcoord{\dx}{0}{1.5} {};
				\node[squarenode] (s-) at \hexcoord{\dx}{-1.5}{4.5} {};
				\node[squarenode] (s+) at \hexcoord{\dx}{-1}{3.5} {};
				
				\draw[undirected] (w) -- (v1);
				\draw[directed, myblue] (v1) -- (v2);
				\draw[directed, myred] (v2) -- (v2p);
				\draw[directed, gray, very thin] (t+) -- (v1);
				\draw[directed, gray, very thin] (v1) -- (t-);
				\draw[directed, gray, very thin] (s+) -- (v2p);
				\draw[directed, gray, very thin] (v2p) -- (s-);
				\draw[directed, gray, very thin] (c+) -- (v2);
				\draw[directed, gray, very thin] (v2) -- (c-);
			\end{scope}
		\end{tikzpicture}\\
		& \begin{tikzpicture}
			
			\node[hexnode] (v1) at \hexcoord{\dx}{0}{0} {$ \sigma $};
			\node[hexnode] (v2) at \hexcoord{\dx}{-1}{2} {$\tilde{\sigma}$};
			\node[hexnode,draw=none,fill=none] (w) at \hexcoord{\dx}{-1}{0} {};
			\node[trianglenode] (t-) at \hexcoord{\dx}{0.5}{0.5} {};
			\node[trianglenode] (t+) at \hexcoord{\dx}{1}{-0.5} {};
			\node[pentanode] (c-) at \hexcoord{\dx}{-0.5}{2.5} {};
			\node[pentanode] (c+) at \hexcoord{\dx}{0}{1.5} {};
			\node[squarenode] (s-) at \hexcoord{\dx}{-1}{0.5} {};
			\node[squarenode] (s+) at \hexcoord{\dx}{-0.5}{-0.5} {};
			
			\draw[undirected] (v1) -- (w);
			\draw[directed, myblue] (v1) -- (v2);
			\draw[directed, gray, very thin] (t+) -- (v1);
			\draw[directed, gray, very thin] (v1) -- (t-);
			\draw[directed, gray, very thin] (s+) -- (v1);
			\draw[directed, gray, very thin] (v1) -- (s-);
			\draw[directed, gray, very thin] (c+) -- (v2);
			\draw[directed, gray, very thin] (v2) -- (c-);
			
			\node[hexnode, draw=none, fill=none] (A1) at (1.5*\dx,1) {};
			\node[hexnode, draw=none, fill=none] (A2) at (2.5*\dx,1) {};
			\draw[bidirected] (A1) -- (A2);
			
			\begin{scope}[shift={(4*\dx,0)}]
				\node[hexnode] (v1) at \hexcoord{\dx}{0}{0} {$ \sigma $};
				\node[hexnode] (v2) at \hexcoord{\dx}{-1}{2} {$\tilde{\sigma}$};
				\node[hexnode] (v2p) at \hexcoord{\dx}{-2}{4} {$\sigma$};
				\node[hexnode,draw=none,fill=none] (w) at \hexcoord{\dx}{-1}{0} {};
				\node[squarenode] (t-) at \hexcoord{\dx}{0.5}{0.5} {};
				\node[squarenode] (t+) at \hexcoord{\dx}{1}{-0.5} {};
				\node[pentanode] (c-) at \hexcoord{\dx}{-0.5}{2.5} {};
				\node[pentanode] (c+) at \hexcoord{\dx}{0}{1.5} {};
				\node[trianglenode] (s-) at \hexcoord{\dx}{-1.5}{4.5} {};
				\node[trianglenode] (s+) at \hexcoord{\dx}{-1}{3.5} {};
				
				\draw[undirected] (w) -- (v1);
				\draw[directed, myred] (v1) -- (v2);
				\draw[directed, myblue] (v2) -- (v2p);
				\draw[directed, gray, very thin] (t+) -- (v1);
				\draw[directed, gray, very thin] (v1) -- (t-);
				\draw[directed, gray, very thin] (s+) -- (v2p);
				\draw[directed, gray, very thin] (v2p) -- (s-);
				\draw[directed, gray, very thin] (c+) -- (v2);
				\draw[directed, gray, very thin] (v2) -- (c-);
			\end{scope}
		\end{tikzpicture}\\
		
	\end{tabular}
	\caption{An $ \text{H}_p $-split, an $ \text{H}_s $-split, and an $ \text{H}_s $-split where the orientation of the two edges swapped. The dashed line indicates the path to the edge $ P((v,w)) $.}\label{fig_Hsplits}
\end{figure}
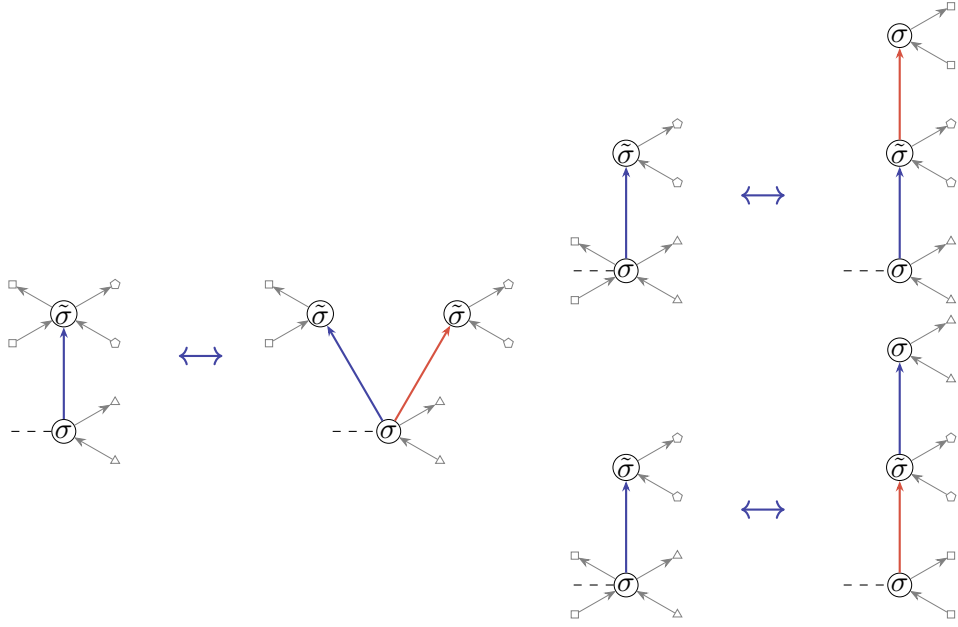
\end{defn}
\def\dx{\dxdefault}
	
	\begin{prop}\label{prop_classTreeMod}
		Let $ J:\S^2\times[0,1]\to\R^3$ be a generic regular homotopy without triple points such that $ J_t $ is generic for all $ t\in[0,1]\setminus \{t_0\} $ and $ J_{t_0}\in I_1 $. Denote $ f:=J_0 $ and $ g:=J_1 $. Then, up to interchanging $ f $ and $ g $, the double point tree of $ G_g $ is constructed from $ G_f $ in one of the following ways
		\begin{enumerate}[label=\roman*)]
			\item \label{enum_treeModE}Type $ E $: 
			$ G_g $ has two additional vertices $ w_1, w_2 $ and two edges $ (v_1,w_1),(v_2,w_2) $ for some $ v_1,v_2\in V_f $ with $ \delta(v_1)-\delta(v_2) \in \{-2,0,2\}$. It holds $ P_g((v_1,w_1))=P_g((v_2,w_2)) $, $ \delta(w_i)=\delta(v_i)\pm 2 $ and $ \abs{\delta(\{v_1,v_2,w_1,w_2\})}=2 $.
			\item \label{enum_treeModH}Type $ H $: 
			For a pair of edges $ a_1=(v_1,w_1),a_2=(v_2,w_2)\in E_f $ with $ P_f(a_1)=a_2 $ and $ \abs{\delta(\{v_1,v_2,w_1,w_2\})}=2 $, the double point tree $ G_g $ is obtained from $ G_f $ by an $ \text{H}_p $- or $ \text{H}_s $-split on $ a_1 $ and an $ \text{H}_p $- or $ \text{H}_s $-split on $ a_2 $. The two added edges define a new pair.
		\end{enumerate} 
	\begin{figure}[H]
	\centering
		\begin{tikzpicture}
			\node[hexnode] (v1) at \hexcoord{\dx}{0}{0} {$ \sigma $};
			\node[hexnode] (w1) at \hexcoord{\dx}{1}{0} {$\tau$};
			\draw[undirected] (v1) -- (w1);
			
			\node[hexnode, draw=none, fill=none] (A1) at (2*\dx,0.5) {};
			\node[hexnode, draw=none, fill=none] (A2) at (3*\dx,0.5) {};
			\draw[bidirected] (A1) -- (A2);

			\begin{scope}[shift={(4*\dx,0)}]
				\node[hexnode] (v1) at \hexcoord{\dx}{0}{0} {$ \sigma $};
				\node[hexnode] (v2) at \hexcoord{\dx}{-0.5}{1} {$\tilde{\sigma}$};
				\node[hexnode] (w1) at \hexcoord{\dx}{1}{0} {$\tau$};
				\node[hexnode] (w2) at \hexcoord{\dx}{0.5}{1} {$ \tilde{\tau} $};
				\draw[directed, myblue] (v1) -- (v2);
				\draw[directed, myblue] (w1) -- (w2);
				\draw[undirected] (v1) -- (w1);
			\end{scope}
			
		\end{tikzpicture}
	\caption{The double point tree modification by a transverse crossing of an $ E $ singularity.}\label{fig_treeModE}
\end{figure}
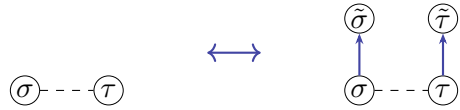
		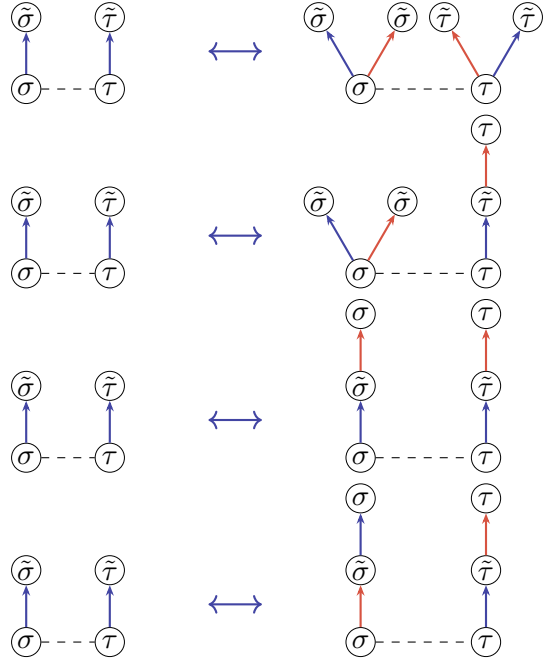
\begin{figure}[H]
			\centering
			\begin{tabular}{l}
			\begin{tikzpicture}
				\node[hexnode] (v1) at \hexcoord{\dx}{0}{0} {$ \sigma $};
				\node[hexnode] (v2) at \hexcoord{\dx}{-0.5}{1} {$\tilde{\sigma}$};
				\node[hexnode] (w1) at \hexcoord{\dx}{1}{0} {$\tau$};
				\node[hexnode] (w2) at \hexcoord{\dx}{0.5}{1} {$ \tilde{\tau} $};
				\draw[directed, myblue] (v1) -- (v2);
				\draw[directed, myblue] (w1) -- (w2);
				\draw[undirected] (v1) -- (w1);
				
				\node[hexnode, draw=none, fill=none] (A1) at (2*\dx,0.5) {};
				\node[hexnode, draw=none, fill=none] (A2) at (3*\dx,0.5) {};
				\draw[bidirected] (A1) -- (A2);

				\begin{scope}[shift={(4*\dx,0)}]
				\node[hexnode] (v1) at \hexcoord{\dx}{0}{0} {$ \sigma $};
				\node[hexnode] (v2) at \hexcoord{\dx}{-1}{1} {$\tilde{\sigma}$};
				\node[hexnode] (v2p) at \hexcoord{\dx}{0}{1} {$\tilde{\sigma}$};
				\node[hexnode] (w1) at \hexcoord{\dx}{1.5}{0} {$\tau$};
				\node[hexnode] (w2) at \hexcoord{\dx}{1.5}{1} {$ \tilde{\tau} $};
				\node[hexnode] (w2p) at \hexcoord{\dx}{0.5}{1} {$ \tilde{\tau} $};
				\draw[directed, myblue] (v1) -- (v2);
				\draw[directed, myred] (v1) -- (v2p);
				\draw[directed, myblue] (w1) -- (w2);
				\draw[directed, myred] (w1) -- (w2p);
				\draw[undirected] (v1) -- (w1);
				\end{scope}
				
			\end{tikzpicture}\\
			\begin{tikzpicture}
				
				\node[hexnode] (v1) at \hexcoord{\dx}{0}{0} {$ \sigma $};
				\node[hexnode] (v2) at \hexcoord{\dx}{-0.5}{1} {$\tilde{\sigma}$};
				\node[hexnode] (w1) at \hexcoord{\dx}{1}{0} {$\tau$};
				\node[hexnode] (w2) at \hexcoord{\dx}{0.5}{1} {$ \tilde{\tau} $};
				\draw[directed, myblue] (v1) -- (v2);
				\draw[directed, myblue] (w1) -- (w2);
				\draw[undirected] (v1) -- (w1);
				
				\node[hexnode, draw=none, fill=none] (A1) at (2*\dx,0.5) {};
				\node[hexnode, draw=none, fill=none] (A2) at (3*\dx,0.5) {};
				\draw[bidirected] (A1) -- (A2);

				\begin{scope}[shift={(4*\dx,0)}]
					\node[hexnode] (v1) at \hexcoord{\dx}{0}{0} {$ \sigma $};
					\node[hexnode] (v2) at \hexcoord{\dx}{-1}{1} {$\tilde{\sigma}$};
					\node[hexnode] (v2p) at \hexcoord{\dx}{0}{1} {$\tilde{\sigma}$};
					\node[hexnode] (w1) at \hexcoord{\dx}{1.5}{0} {$\tau$};
					\node[hexnode] (w2) at \hexcoord{\dx}{1}{1} {$ \tilde{\tau} $};
					\node[hexnode] (w2p) at \hexcoord{\dx}{0.5}{2} {$ \tau $};
					\draw[directed, myblue] (v1) -- (v2);
					\draw[directed, myred] (v1) -- (v2p);
					\draw[directed, myblue] (w1) -- (w2);
					\draw[directed, myred] (w2) -- (w2p);
					\draw[undirected] (v1) -- (w1);
				\end{scope}
				
			\end{tikzpicture}\\
			\begin{tikzpicture}
				
				\node[hexnode] (v1) at \hexcoord{\dx}{0}{0} {$ \sigma $};
				\node[hexnode] (v2) at \hexcoord{\dx}{-0.5}{1} {$\tilde{\sigma}$};
				\node[hexnode] (w1) at \hexcoord{\dx}{1}{0} {$\tau$};
				\node[hexnode] (w2) at \hexcoord{\dx}{0.5}{1} {$ \tilde{\tau} $};
				\draw[directed, myblue] (v1) -- (v2);
				\draw[directed, myblue] (w1) -- (w2);
				\draw[undirected] (v1) -- (w1);
				
				\node[hexnode, draw=none, fill=none] (A1) at (2*\dx,0.5) {};
				\node[hexnode, draw=none, fill=none] (A2) at (3*\dx,0.5) {};
				\draw[bidirected] (A1) -- (A2);

				\begin{scope}[shift={(4*\dx,0)}]
					\node[hexnode] (v1) at \hexcoord{\dx}{0}{0} {$ \sigma $};
					\node[hexnode] (v2) at \hexcoord{\dx}{-0.5}{1} {$\tilde{\sigma}$};
					\node[hexnode] (v2p) at \hexcoord{\dx}{-1}{2} {$\sigma$};
					\node[hexnode] (w1) at \hexcoord{\dx}{1.5}{0} {$\tau$};
					\node[hexnode] (w2) at \hexcoord{\dx}{1}{1} {$ \tilde{\tau} $};
					\node[hexnode] (w2p) at \hexcoord{\dx}{0.5}{2} {$ \tau $};
					\draw[directed, myblue] (v1) -- (v2);
					\draw[directed, myred] (v2) -- (v2p);
					\draw[directed, myblue] (w1) -- (w2);
					\draw[directed, myred] (w2) -- (w2p);
					\draw[undirected] (v1) -- (w1);
				\end{scope}
				
			\end{tikzpicture}\\
			\begin{tikzpicture}
				
				\node[hexnode] (v1) at \hexcoord{\dx}{0}{0} {$ \sigma $};
				\node[hexnode] (v2) at \hexcoord{\dx}{-0.5}{1} {$\tilde{\sigma}$};
				\node[hexnode] (w1) at \hexcoord{\dx}{1}{0} {$\tau$};
				\node[hexnode] (w2) at \hexcoord{\dx}{0.5}{1} {$ \tilde{\tau} $};
				\draw[directed, myblue] (v1) -- (v2);
				\draw[directed, myblue] (w1) -- (w2);
				\draw[undirected, black, dashed, thin] (v1) -- (w1);
				
				\node[hexnode, draw=none, fill=none] (A1) at (2*\dx,0.5) {};
				\node[hexnode, draw=none, fill=none] (A2) at (3*\dx,0.5) {};
				\draw[bidirected] (A1) -- (A2);

				\begin{scope}[shift={(4*\dx,0)}]
					\node[hexnode] (v1) at \hexcoord{\dx}{0}{0} {$ \sigma $};
					\node[hexnode] (v2) at \hexcoord{\dx}{-0.5}{1} {$\tilde{\sigma}$};
					\node[hexnode] (v2p) at \hexcoord{\dx}{-1}{2} {$\sigma$};
					\node[hexnode] (w1) at \hexcoord{\dx}{1.5}{0} {$\tau$};
					\node[hexnode] (w2) at \hexcoord{\dx}{1}{1} {$ \tilde{\tau} $};
					\node[hexnode] (w2p) at \hexcoord{\dx}{0.5}{2} {$ \tau $};
					\draw[directed, myred] (v1) -- (v2);
					\draw[directed, myblue] (v2) -- (v2p);
					\draw[directed, myblue] (w1) -- (w2);
					\draw[directed, myred] (w2) -- (w2p);
					\draw[undirected] (v1) -- (w1);
				\end{scope}
				
			\end{tikzpicture}
			\end{tabular}
			\caption{The double point tree modifications by a transverse crossing of an $ H $ singularity.}\label{fig_treeModH}
		\end{figure}
	\end{prop}
	\begin{proof}
		The different types arise from the case distinction $ J_{t_0}\in E $ and $ J_{t_0}\in H $.\\
		\ref{enum_treeModE} This is straight-forward from considering the preimages of the local picture of an $ E $ singularity (see Figure~\ref{fig_EHTQ}). The restrictions on the degrees is obtained from studying the degrees of the components in this local picture for both cases of orientation.\\
		\ref{enum_treeModH} From Lemma~\ref{lem_Hconnectivity} (and Figure~\ref{fig_Hconnectivity}) we deduce the tree changes with specified degrees $ \sigma,\tilde{\sigma},\tau,\tilde{\tau} $ depicted in Figure~\ref{fig_HtreeMod}.
		Consider the two connected subtrees $ C_1 $ and $ C_2 $ with three vertices each, when resolving the $ H $ singularity to two double point curves instead of one (see Figure~\ref{fig_HtreeMod}). Then, there exists $ v_1\in C_1 $ which lies closest to $ C_2 $, and vice versa, there exists $ v_2\in C_2 $ closest to $ C_1 $. It may be that $ v_1=v_2 $ or that $ v_1 $ is joined to $ v_2 $ by an edge or a path of multiple edges (i.e.\ other double point curves encapsulating one of the two figure eights). For each choice of $ v_i $ the edges in $ C_i $ are directed away from $ v_i $ (see Figure~\ref{fig_HtreeModCases}).
		\begin{figure}[H]
			\centering
			\begin{tikzpicture}
				\node[hexnode] (w1) at \hexcoord{\dx}{0}{0} {$ \sigma $};
				\node[hexnode] (w2) at \hexcoord{\dx}{-0.5}{1} {$\tilde{\sigma}$};
				\node[hexnode] (z1) at \hexcoord{\dx}{1}{0} {$ \tau $};
				\node[hexnode] (z2) at \hexcoord{\dx}{0.5}{1} {$\tilde{\tau}$};
				
				\draw[undirected, solid, thick, myblue] (w1) -- (w2);
				\draw[undirected, solid, thick, myblue] (z1) -- (z2);
				
				\node[hexnode, draw=none, fill=none] (A1) at (2*\dx,0.5) {};
				\node[hexnode, draw=none, fill=none] (A2) at (3*\dx,0.5) {};
				\draw[bidirected] (A1) -- (A2);
				
				\begin{scope}[shift={(4*\dx,0.5)}]
				\node[hexnode] (w1) at \hexcoord{\dx}{0}{0} {$ \sigma $};
				\node[hexnode] (w2) at \hexcoord{\dx}{-1}{1} {$\tilde{\sigma}$};
				\node[hexnode] (w3) at \hexcoord{\dx}{0}{-1} {$\tilde{\sigma}$};
				\node[hexnode] (z1) at \hexcoord{\dx}{1.5}{0} {$ \tau $};
				\node[hexnode] (z2) at \hexcoord{\dx}{1.5}{1} {$\tilde{\tau}$};
				\node[hexnode] (z3) at \hexcoord{\dx}{2.5}{-1} {$\tilde{\tau}$};
				
				\draw[undirected, thick, solid, myblue] (w1) -- (w3);
				\draw[undirected, thick, solid, myred] (w1) -- (w2);
				
				\draw[undirected, thick, solid, myblue] (z1) -- (z3);
				\draw[undirected, thick, solid, myred] (z1) -- (z2);
				\end{scope}
				
			\end{tikzpicture}
		\caption{The four (resp.\ six) vertices taking part in the $ H $ singularity.}\label{fig_HtreeMod}
		\end{figure}
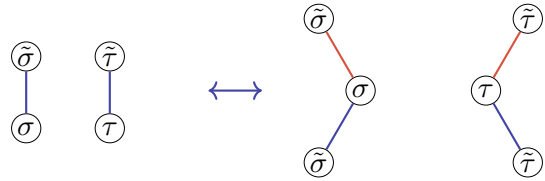
		\def\dx{0.9}
		\begin{figure}[H]
			\centering
			\begin{tabular}{c c}
				\subfigure[]{
					\begin{minipage}[c][0.3\textwidth][c]{0.375\textwidth}
						\centering
					\begin{tabular}{c c}
						\begin{tikzpicture}
							\node[hexnode] (w1) at \hexcoord{\dx}{0}{0} {$ \sigma $};
							\node[hexnode] (w2) at \hexcoord{\dx}{-1}{1} {$\tilde{\sigma}$};
							\node[hexnode] (w3) at \hexcoord{\dx}{0}{-1} {$\tilde{\sigma}$};
							\node[hexnode] (z1) at \hexcoord{\dx}{1.5}{0} {$ \tau $};
							\node[hexnode] (z2) at \hexcoord{\dx}{1.5}{1} {$\tilde{\tau}$};
							\node[hexnode] (z3) at \hexcoord{\dx}{2.5}{-1} {$\tilde{\tau}$};
							
							\draw[directed, myblue] (w1) -- (w3);
							\draw[directed, thick, solid, myred] (w1) -- (w2);
							
							\draw[directed, myblue] (z1) -- (z3);
							\draw[directed, myred] (z1) -- (z2);
							
							\draw[undirected] (w1) -- (z1);
							
						\end{tikzpicture} &
						\begin{tikzpicture}
							\node[hexnode] (w1) at \hexcoord{\dx}{0}{0} {$ \sigma $};
							\node[hexnode] (w2) at \hexcoord{\dx}{-1}{1} {$\tilde{\sigma}$};
							\node[hexnode] (w3) at \hexcoord{\dx}{0}{-1} {$\tilde{\sigma}$};
							\node[hexnode] (z1) at \hexcoord{\dx}{1.5}{0} {$   \tau$};
							\node[hexnode] (z2) at \hexcoord{\dx}{1.5}{1} {$ \tilde \tau$};
							\node[hexnode] (z3) at \hexcoord{\dx}{2.5}{-1} {$\tilde \tau$};
							
							\draw[directed, myblue] (w3) -- (w1);
							\draw[directed, myred] (w1) -- (w2);
							
							\draw[directed, myblue] (z3) -- (z1);
							\draw[directed, myred] (z1) -- (z2);
							
							\draw[undirected] (w3) -- (z3);
							
						\end{tikzpicture} \\
						\begin{tikzpicture}
							\node[hexnode] (w1) at \hexcoord{\dx}{0}{0} {$ \sigma $};
							\node[hexnode] (w2) at \hexcoord{\dx}{-1}{1} {$\tilde{\sigma}$};
							\node[hexnode] (w3) at \hexcoord{\dx}{0}{-1} {$\tilde{\sigma}$};
							\node[hexnode] (z1) at \hexcoord{\dx}{1.5}{0} {$ \tau $};
							\node[hexnode] (z2) at \hexcoord{\dx}{1.5}{1} {$\tilde{\tau}$};
							\node[hexnode] (z3) at \hexcoord{\dx}{2.5}{-1} {$\tilde{\tau}$};
							
							\draw[directed, myblue] (w3) -- (w1);
							\draw[directed, myred] (w1) -- (w2);
							
							\draw[directed, myblue] (z1) -- (z3);
							\draw[directed, myred] (z1) -- (z2);
							
							\draw[undirected] (w3) -- (z1);
							
						\end{tikzpicture} &
						\begin{tikzpicture}
							\node[hexnode] (w1) at \hexcoord{\dx}{0}{0} {$ \sigma $};
							\node[hexnode] (w2) at \hexcoord{\dx}{-1}{1} {$\tilde{\sigma}$};
							\node[hexnode] (w3) at \hexcoord{\dx}{0}{-1} {$\tilde{\sigma}$};
							\node[hexnode] (z1) at \hexcoord{\dx}{1.5}{0} {$ \tau $};
							\node[hexnode] (z2) at \hexcoord{\dx}{1.5}{1} {$\tilde{\tau}$};
							\node[hexnode] (z3) at \hexcoord{\dx}{2.5}{-1} {$\tilde{\tau}$};
							
							\draw[directed, myblue] (w3) -- (w1);
							\draw[directed, myred] (w1) -- (w2);
							
							\draw[directed, myblue] (z1) -- (z3);
							\draw[directed, myred] (z2) -- (z1);
							\draw[undirected] (w3) -- (z2);
							
						\end{tikzpicture} 
					\end{tabular}		
					\end{minipage}
				}
				&
				\subfigure[]{
					\begin{minipage}[c][0.3\textwidth][c]{0.525\textwidth}
						\centering
						\includegraphics[width=\textwidth]{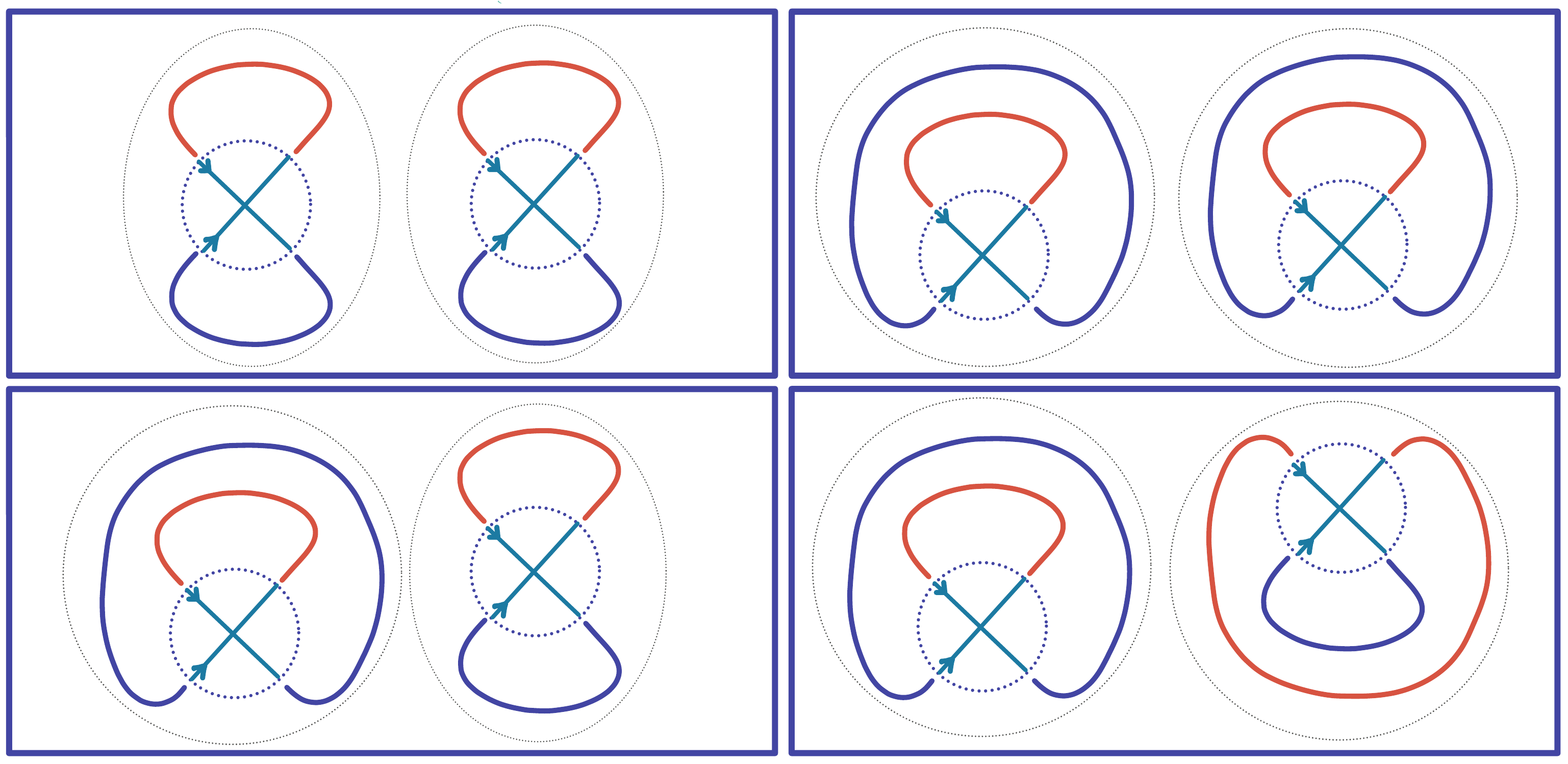}
					\end{minipage}
				}
			\end{tabular}
			\caption{The different cases of the modification of an $ H $ singularity on the double point tree.}\label{fig_HtreeModCases}
		\end{figure}
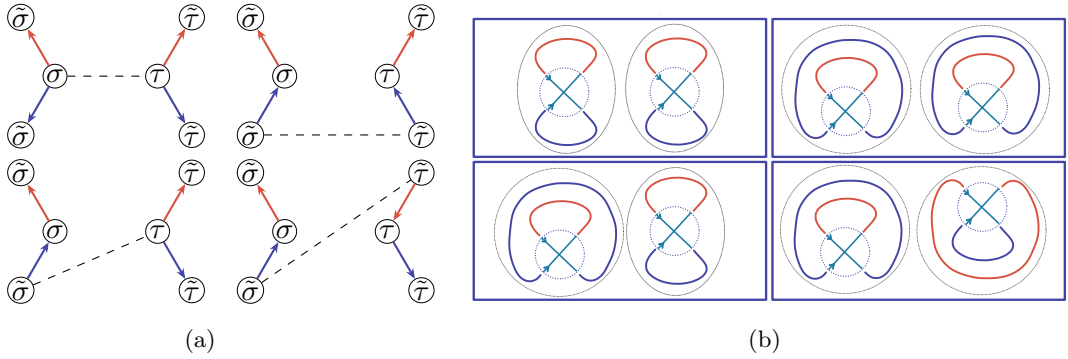
		We arrive at the above described modifications displayed in Figure~\ref{fig_treeModH}.
		The components of degree $ \tilde \sigma $ and $ \tilde \tau $ split into two components each. Any double point curve originally contained within such a component, is afterwards contained in one of the two new components.
		The restriction $ \abs{\delta(\{v_1,v_2,w_1,w_2\})}=2 $ on the degrees, again, comes from studying the local picture of an $ H $ singularity (see Figure~\ref{fig_EHTQ}) for both cases of orientation, $ H^1 $ and $ H^2 $.
	\end{proof}
	
	\begin{rem}
		Every modification of a double point tree as described in Proposition~\ref{prop_classTreeMod} results in a new double point tree. However, the resulting tree is not necessarily realizable (see Figure~\ref{fig_exTreeModNotAllowed} for an example, verifiable via the classification of \cite{Lippner}).
		Nonetheless, for each of the four cases in Figure~\ref{fig_treeModH} one can find regular homotopies realizing the corresponding tree modification.
	\end{rem}
	\begin{figure}[H]
		\centering
		\begin{tikzpicture}
			
			\node[hexnode] (v1) at \hexcoord{\dx}{0}{0} {$ 1 $};
			\node[hexnode] (w1) at \hexcoord{\dx}{0}{1} {$0$};
			\node[hexnode] (w2) at \hexcoord{\dx}{-1}{1} {$0$};
			\draw[directed, myblue] (v1) -- (w1);
			\draw[directed, myblue] (v1) -- (w2);
			\draw[undirected] (v1) -- (w1);
			
			\node[hexnode, draw=none, fill=none] (A1) at (1*\dx,0.5) {};
			\node[hexnode, draw=none, fill=none] (A2) at (2*\dx,0.5) {};
			\node[anchor=south] at (1.5*\dx,0.5) {$ H $};
			\draw[bidirected,->] (A1) -- (A2);

			\begin{scope}[shift={(3*\dx,0)}]
				\node[hexnode] (v1) at \hexcoord{\dx}{0}{0} {$ 1 $};
				\node[hexnode] (w1) at \hexcoord{\dx}{0}{1} {$0$};
				\node[hexnode] (w2) at \hexcoord{\dx}{-1}{1} {$0$};
				\node[hexnode] (z1) at \hexcoord{\dx}{0}{2} {$1$};
				\node[hexnode] (z2) at \hexcoord{\dx}{-2}{2} {$ 1 $};
				\draw[directed, myred] (v1) -- (w1);
				\draw[directed, myblue] (v1) -- (w2);
				\draw[directed, myblue] (w1) -- (z1);
				\draw[directed, myred] (w2) -- (z2);
			\end{scope}
			
		\end{tikzpicture}
		\caption{A non-realizable double point tree obtained from a realizable double point tree by a modification of Proposition~\ref{prop_classTreeMod}.}\label{fig_exTreeModNotAllowed}
	\end{figure}
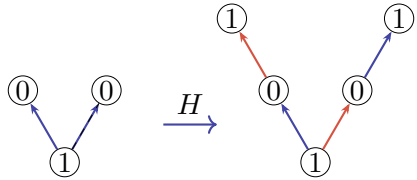
	
	\section{Proof of Theorem~\ref{thm_Invariant}}\label{sec_Invariant}
	\begin{proof}[Proof of Theorem~\ref{thm_Invariant}]
		As discussed in Section~\ref{sec_prelim}, one can perturb the regular homotopy to a generic one and thus it suffices to show that $ F $ is unchanged by transverse crossings of the strata of types E and H. Let $ f,g\in \Imm_{<3}(\S^2,\R^3) $ be generic immersions such that they are joined by a generic regular homotopy with exactly one singular time as in Proposition~\ref{prop_classTreeMod}. We make a case distinction by type of the singularity.
		
		Type E:\\
		The added (resp.\ removed) vertices have indegrees of $ 1 $ and the indegrees of all other vertices remain unchanged. Therefore, $ F(f) = F(g) $.
		
		Type H:\\
		\begin{figure}[H]
			\centering
				\def\dx{0.9}
				\begin{tikzpicture}
					
					\node at (0,-1*\dx) {$ G_f $};
					
					\node[hexnode] (v1) at \hexcoord{\dx}{0}{0} {$ \sigma $};
					\node[hexnode] (v2) at \hexcoord{\dx}{-1}{2} {$\tilde{\sigma}$};
					\node[hexnode,draw=none,fill=none] (w) at \hexcoord{\dx}{-1}{0} {};
					\node[trianglenode] (t-) at \hexcoord{\dx}{0.5}{0.5} {};
					\node[trianglenode] (t+) at \hexcoord{\dx}{1}{-0.5} {};
					\node[pentanode] (c-) at \hexcoord{\dx}{-0.5}{2.5} {};
					\node[pentanode] (c+) at \hexcoord{\dx}{0}{1.5} {};
					\node[squarenode] (s-) at \hexcoord{\dx}{-2}{2.5} {};
					\node[squarenode] (s+) at \hexcoord{\dx}{-1.5}{1.5} {};
					
					\draw[undirected] (v1) -- (w);
					\draw[directed, myblue] (v1) -- (v2);
					\draw[directed, gray, very thin] (t+) -- (v1);
					\draw[directed, gray, very thin] (v1) -- (t-);
					\draw[directed, gray, very thin] (s+) -- (v2);
					\draw[directed, gray, very thin] (v2) -- (s-);
					\draw[directed, gray, very thin] (c+) -- (v2);
					\draw[directed, gray, very thin] (v2) -- (c-);
					
					\node[highlight] (v2H) at \hexcoord{\dx}{-1}{2} {};
					\node[anchor=north west,myred,minimum size=8mm] (v2Hl) at \hexcoord{\dx}{-1}{2} {$ v_f $};
					
					\node[hexnode, draw=none, fill=none] (A1) at (1.5*\dx,1) {};
					\node[hexnode, draw=none, fill=none] (A2) at (2.5*\dx,1) {};
					\draw[bidirected] (A1) -- (A2);
					
					\begin{scope}[shift={(4.75*\dx,0)}]
						\node at (0,-1*\dx) {$ G_g $};
						
						\node[hexnode] (v1) at \hexcoord{\dx}{0}{0} {$ \sigma $};
						\node[hexnode] (v2) at \hexcoord{\dx}{-2}{2} {$\tilde{\sigma}$};
						\node[hexnode] (v2p) at \hexcoord{\dx}{0}{2} {$\tilde{\sigma}$};
						\node[hexnode,draw=none,fill=none] (w) at \hexcoord{\dx}{-1}{0} {};
						\node[trianglenode] (t-) at \hexcoord{\dx}{0.5}{0.5} {};
						\node[trianglenode] (t+) at \hexcoord{\dx}{1}{-0.5} {};
						\node[pentanode] (c-) at \hexcoord{\dx}{0.5}{2.5} {};
						\node[pentanode] (c+) at \hexcoord{\dx}{1}{1.5} {};
						\node[squarenode] (s-) at \hexcoord{\dx}{-3}{2.5} {};
						\node[squarenode] (s+) at \hexcoord{\dx}{-2.5}{1.5} {};
						
						\draw[undirected] (w) -- (v1);
						\draw[directed, myblue] (v1) -- (v2);
						\draw[directed, myred] (v1) -- (v2p);
						\draw[directed, gray, very thin] (t+) -- (v1);
						\draw[directed, gray, very thin] (v1) -- (t-);
						\draw[directed, gray, very thin] (s+) -- (v2);
						\draw[directed, gray, very thin] (v2) -- (s-);
						\draw[directed, gray, very thin] (c+) -- (v2p);
						\draw[directed, gray, very thin] (v2p) -- (c-);
						
						\node[highlight] (v2H) at \hexcoord{\dx}{-2}{2} {};
						\node[anchor=south west,myred,minimum size=8mm] (v2Hl) at (v2H) {$ v_g $};
						\node[highlight] (v2pH) at \hexcoord{\dx}{0}{2} {};
						\node[anchor=south east,myred,minimum size=8mm] (v2pHl) at (v2pH) {$ w_g $};
					\end{scope}
				
					\begin{scope}[shift={(9*\dx,0)}]
					\node at (0,-1*\dx) {$ G_f $};
					
					\node[hexnode] (v1) at \hexcoord{\dx}{0}{0} {$ \sigma $};
					\node[hexnode] (v2) at \hexcoord{\dx}{-1}{2} {$\tilde{\sigma}$};
					\node[hexnode,draw=none,fill=none] (w) at \hexcoord{\dx}{-1}{0} {};
					\node[trianglenode] (t-) at \hexcoord{\dx}{0.5}{0.5} {};
					\node[trianglenode] (t+) at \hexcoord{\dx}{1}{-0.5} {};
					\node[pentanode] (c-) at \hexcoord{\dx}{-0.5}{2.5} {};
					\node[pentanode] (c+) at \hexcoord{\dx}{0}{1.5} {};
					\node[squarenode] (s-) at \hexcoord{\dx}{-1}{0.5} {};
					\node[squarenode] (s+) at \hexcoord{\dx}{-0.5}{-0.5} {};
					
					\draw[undirected] (v1) -- (w);
					\draw[directed, myblue] (v1) -- (v2);
					\draw[directed, gray, very thin] (t+) -- (v1);
					\draw[directed, gray, very thin] (v1) -- (t-);
					\draw[directed, gray, very thin] (s+) -- (v1);
					\draw[directed, gray, very thin] (v1) -- (s-);
					\draw[directed, gray, very thin] (c+) -- (v2);
					\draw[directed, gray, very thin] (v2) -- (c-);
					
					\node[hexnode, draw=none, fill=none] (A1) at (1.5*\dx,1) {};
					\node[hexnode, draw=none, fill=none] (A2) at (2.5*\dx,1) {};
					\draw[bidirected] (A1) -- (A2);
					
					\node[highlight] (v1H) at \hexcoord{\dx}{0}{0} {};
					\node[anchor=north,myred,minimum size=10mm] (v1Hl) at (v1H) {$ v_f $};
					
					\begin{scope}[shift={(4*\dx,0)}]
						\node at (0,-1*\dx) {$ G_g $};
						
						\node[hexnode] (v1) at \hexcoord{\dx}{0}{0} {$ \sigma $};
						\node[hexnode] (v2) at \hexcoord{\dx}{-1}{2} {$\tilde{\sigma}$};
						\node[hexnode] (v2p) at \hexcoord{\dx}{-2}{4} {$\sigma$};
						\node[hexnode,draw=none,fill=none] (w) at \hexcoord{\dx}{-1}{0} {};
						\node[trianglenode] (t-) at \hexcoord{\dx}{0.5}{0.5} {};
						\node[trianglenode] (t+) at \hexcoord{\dx}{1}{-0.5} {};
						\node[pentanode] (c-) at \hexcoord{\dx}{-0.5}{2.5} {};
						\node[pentanode] (c+) at \hexcoord{\dx}{0}{1.5} {};
						\node[squarenode] (s-) at \hexcoord{\dx}{-1.5}{4.5} {};
						\node[squarenode] (s+) at \hexcoord{\dx}{-1}{3.5} {};
						
						\draw[undirected] (w) -- (v1);
						\draw[directed, myblue] (v1) -- (v2);
						\draw[directed, myred] (v2) -- (v2p);
						\draw[directed, gray, very thin] (t+) -- (v1);
						\draw[directed, gray, very thin] (v1) -- (t-);
						\draw[directed, gray, very thin] (s+) -- (v2p);
						\draw[directed, gray, very thin] (v2p) -- (s-);
						\draw[directed, gray, very thin] (c+) -- (v2);
						\draw[directed, gray, very thin] (v2) -- (c-);
						
						\node[highlight] (v1H) at \hexcoord{\dx}{0}{0} {};
						\node[anchor=north east,myred,minimum size=8mm] (v1Hl) at (v1H) {$ v_g $};
						\node[highlight] (v2pH) at \hexcoord{\dx}{-2}{4} {};
						\node[anchor=north east,myred,minimum size=8mm] (v2pHl) at (v2pH) {$ w_g $};
					\end{scope}
				\end{scope}
			\end{tikzpicture}
			\def\dx{\dxdefault}
			\caption{The $ H_p $- and $ \text{H}_s $-split on $ G_f $.}\label{fig_invariantH}
		\end{figure}
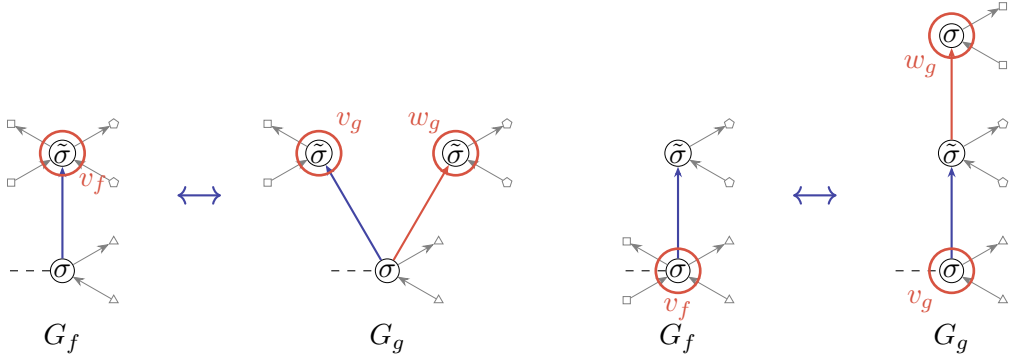
		We look at one half of the H singularity where the classification in Section~\ref{sec_treeEH} yields two cases when ignoring pairing (see Figure~\ref{fig_Hsplits}): an $ H_p $-split or an $ H_s $-split. In both cases the indegrees of $ G_f $ and $ G_g $ differ only in vertices $ v_f,v_g,w_g $ (see Figure~\ref{fig_invariantH}). They share the same topological degree and satisfy
		\begin{align*}
			\deg^-(v_f) &= \deg^-(v_g) + \deg^-(w_g) -1 \\
			\iff 1 - \deg^-(v_f) &= 1 - \deg^-(v_g)  + 1 - \deg^-(w_g).
		\end{align*}
		Therefore, we have $ F(f) = F(g) $.
	\end{proof}
	\begin{prop}\label{prop_extension}
		The invariant $ F $ extends uniquely onto $ \Imm_{<3}(\S^2,\R^3) $ such that $ F $ is invariant under regular homotopies without triple points.
	\end{prop}
	\begin{proof}
		Let $ f\in \Imm(\S^2,\R^3) $ be without triple points. Then, in $ C^k $-topology a neighborhood of $ f $ does not have triple points and a slight perturbation $ \tilde{f} $ of $ f $ yields a generic immersion with $ F(\tilde f) $ defined. We set $ F(f):=F(\tilde f) $. This is well defined since any two generic perturbations of $ f $ are joined by a regular homotopy without triple points. Furthermore, this is the unique extended invariant since $ \tilde{f} $ was obtained from $ f $ by a regular homotopy without triple points.
	\end{proof}
	We denote by $ e_i:=(\delta_{ik})_{k\in\Z} $ the standard basis vectors of $ \Z^\Z $.
	\begin{ex}\label{ex_Fej}
		Let $ j $ be as in Figure~\ref{fig_sketchJ} with outer normal and recall the standard embedding $ e:\S^2\to\R^3 $. We compute directly with Lemma~\ref{lem_delta-f} and Examples~\ref{ex_eTree} and~\ref{ex_jTree}
		\begin{align*}
			F(-j) = e_{\m 3},\quad F(-e) = e_{\m 1},\quad F(e)=e_1,\quad F(j) = e_3.
		\end{align*}
	\end{ex}
	
	\section{Proof of Theorem~\ref{thm_imageF}}\label{sec_image}
	The invariant $ F $ implies that $ \Imm_{<3}(\S^2,\R^3) $ has many more components than the four indicated in Example~\ref{ex_Fej}. We give some examples in Figure~\ref{fig_Fothers} and Examples~\ref{ex_Fek} and~\ref{ex_F2e1-ek} that help to determine the image of $ F $.
	\begin{ex}\label{ex_Fek}
		For all odd $ k\in \Z $, there exist surfaces of revolution $ f_k\in \Imm_{<3}(\S^2,\R^3) $ with $ F(f_k)=e_{k} $  (see Figure~\ref{fig_Fek}).
	\end{ex}
	
	\begin{ex}\label{ex_F2e1-ek}
		For all odd $ k\in \Z $, there exist surfaces of revolution $ g_k\in \Imm_{<3}(\S^2,\R^3) $ with $ F(g_k)=2e_1 - e_{k} $ (see Figure~\ref{fig_F2e1-ek}).
	\end{ex}
	\begin{figure}[H]
		\centering
		\begin{tikzpicture}
			\node[anchor=south] at (0,0.4) {\includegraphics[width=0.95\textwidth]{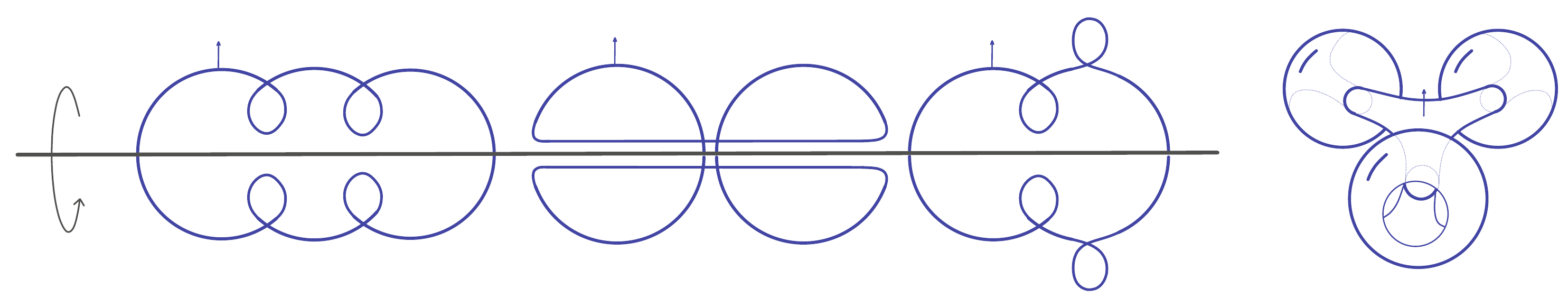}};
			\def\dd{0.25\textwidth}
			
			\foreach \k/\label in {-1.15/$F=2e_3-e_1$, -0.2/$F=2e_1 - e_{\m 1}$, 0.65/$F=e_3+e_{\m 1}-e_{1}$, 1.55/$F=3e_{\m 1}-2e_1$} {
				\node[anchor=south] at (\k*\dd,-2) {\small \label};
			}
			
			\def\dx{\dxdefault}

			\begin{scope}[shift={(-1.15*\dd,0)}]
				\node[hexnode] (v1) at \hexcoord{\dx}{0}{-1} {$ 3 $};
				\node[hexnode] (v2) at \hexcoord{\dx}{1}{-1} {$ 3 $};
				\node[hexnode] (w1) at \hexcoord{\dx}{-1}{0} {$1$};
				\node[hexnode] (w2) at \hexcoord{\dx}{0}{0} {$1$};
				\node[hexnode] (w3) at \hexcoord{\dx}{1}{0} {$1$};
				\draw[directed, myred] (v1) -- (w1);
				\draw[directed, myred] (v1) -- (w2);
				\draw[directed, myblue] (v2) -- (w2);
				\draw[directed, myblue] (v2) -- (w3);
			\end{scope}
			
			\begin{scope}[shift={(-0.2*\dd,0)}]
				\node[hexnode] (v1) at \hexcoord{\dx}{0}{-1} {$ 1 $};
				\node[hexnode] (v2) at \hexcoord{\dx}{1}{-1} {$ 1 $};
				\node[hexnode] (w1) at \hexcoord{\dx}{-1}{0} {$\m 1$};
				\node[hexnode] (w2) at \hexcoord{\dx}{0}{0} {$\m 1$};
				\node[hexnode] (w3) at \hexcoord{\dx}{1}{0} {$\m 1$};
				\draw[directed, myred] (v1) -- (w1);
				\draw[directed, myred] (v1) -- (w2);
				\draw[directed, myblue] (v2) -- (w2);
				\draw[directed, myblue] (v2) -- (w3);
			\end{scope}
			
			\begin{scope}[shift={(0.65*\dd,0)}]
				\node[hexnode] (v1) at \hexcoord{\dx}{0}{-1} {$ 3 $};
				\node[hexnode] (v2) at \hexcoord{\dx}{1}{-1} {$ \m 1 $};
				\node[hexnode] (w1) at \hexcoord{\dx}{-1}{0} {$1$};
				\node[hexnode] (w2) at \hexcoord{\dx}{0}{0} {$1$};
				\node[hexnode] (w3) at \hexcoord{\dx}{1}{0} {$1$};
				\draw[directed, myred] (v1) -- (w1);
				\draw[directed, myred] (v1) -- (w2);
				\draw[directed, myblue] (v2) -- (w2);
				\draw[directed, myblue] (v2) -- (w3);
			\end{scope}

			\def\dx{0.75}
			
			\begin{scope}[shift={(1.5*\dd,-0.1*\dd)}]
				\node[hexnode] (v) at \hexcoord{\dx}{0}{0} {\tiny $ 1 $};
				\node[hexnode] (w1) at \hexcoord{\dx}{-1}{1} {\tiny $ \m 1 $};
				\node[hexnode] (w2) at \hexcoord{\dx}{-2}{1} {\tiny $ 1$};
				\node[hexnode] (y1) at \hexcoord{\dx}{0}{-1} {\tiny $\m 1$};
				\node[hexnode] (y2) at \hexcoord{\dx}{-1}{-1} {\tiny $ 1$};
				\node[hexnode] (z1) at \hexcoord{\dx}{1}{0} {\tiny $ \m 1$};
				\node[hexnode] (z2) at \hexcoord{\dx}{2}{0} {\tiny $ 1$};
				\draw[directed, myred] (w1) -- (v);
				\draw[directed, myred] (w1) -- (w2);
				\draw[directed, myblue] (y1) -- (v);
				\draw[directed, myblue] (y1) -- (y2);
				\draw[directed, mygreen] (z1) -- (v);
				\draw[directed, mygreen] (z1) -- (z2);
			\end{scope}
		\end{tikzpicture}
		\caption{Some immersions $ f\in \Imm_{<3}(\S^2,\R^3) $ with their values of $ F $.}\label{fig_Fothers}
	\end{figure}
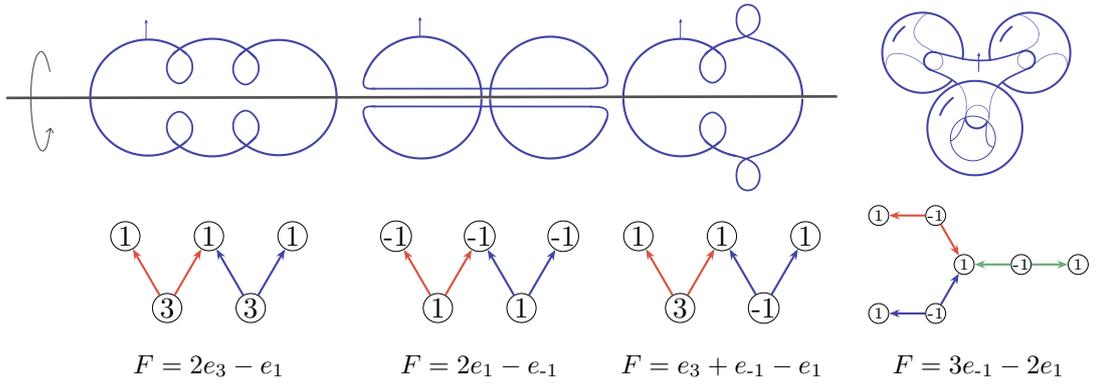
	\def\dx{\dxdefault}
	
	\begin{figure}[H]
		\centering
		\begin{tikzpicture}
			\def\dd{0.152\textwidth};
			
			\node[anchor=south] at (0,1.65*\dd) {\includegraphics[width=0.25\textwidth]{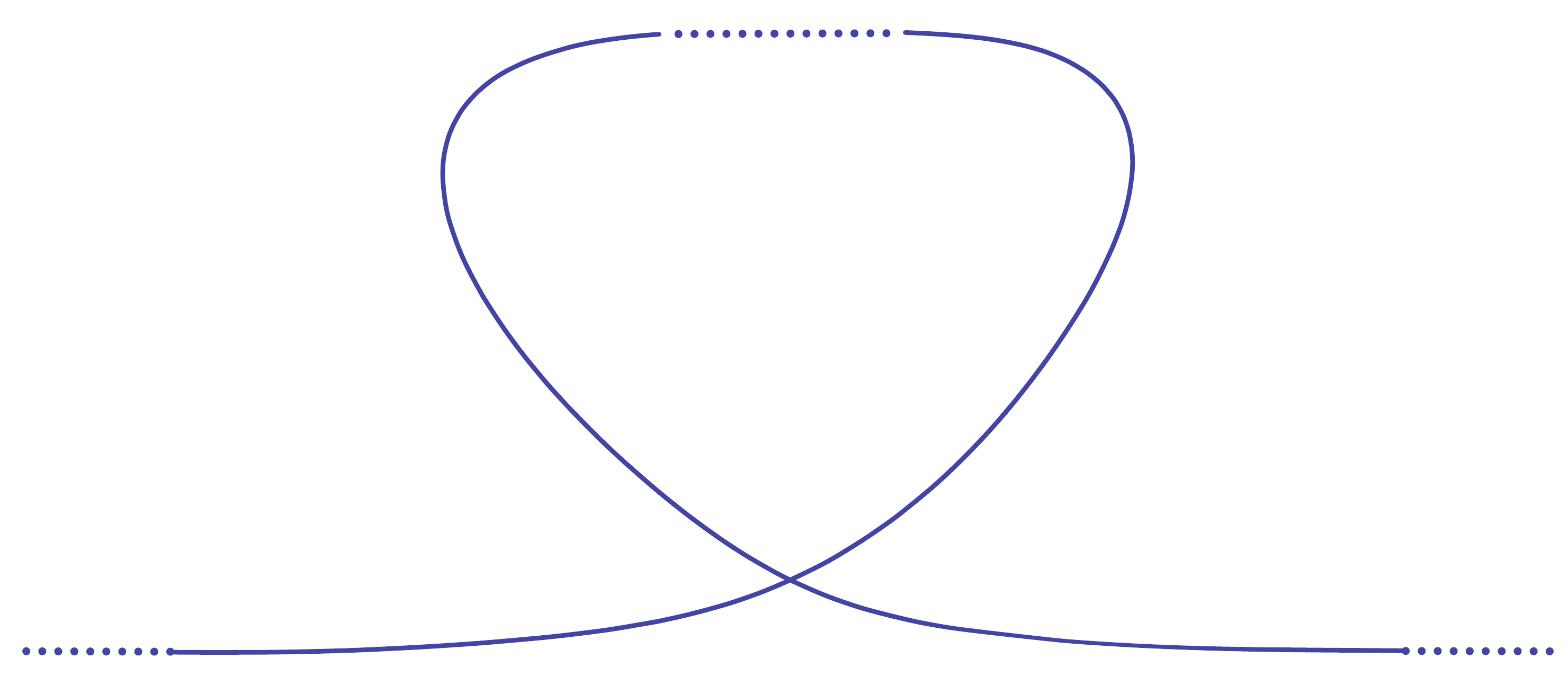}};
			\node[anchor=south] at (0,0) {\includegraphics[width=0.9\textwidth]{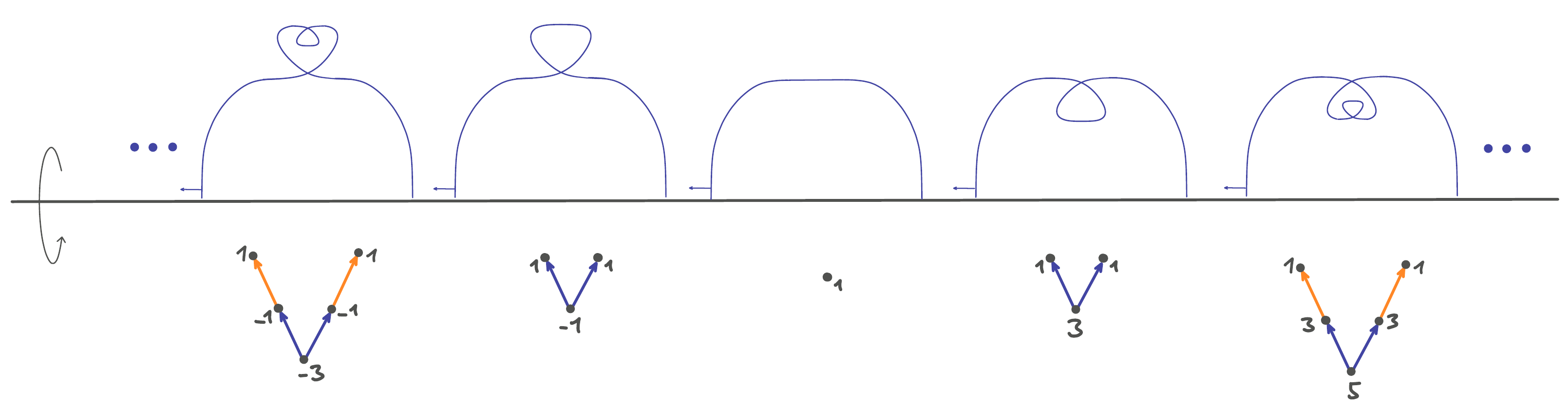}};
			
			\foreach \k/\label in {-2.1/$F=e_{\m 3}$, -1/$e_{\m 1}$, 0/$e_1$, 1/$e_3$, 2/$ e_5 $} {
				\node[anchor=north west] at (\k*\dd,0) {\small \label};
			}
		\end{tikzpicture}
		\caption{The fundamental piece to construct $ f_k $ and some examples for $ k=-3,-1,1,3,5 $.}\label{fig_Fek}
	\end{figure}
	
	\begin{figure}[H]
		\centering
		\begin{tikzpicture}
			\def\dd{0.15\textwidth};
			
			\node[anchor=south] at (0,1.65*\dd) {\includegraphics[width=0.23\textwidth]{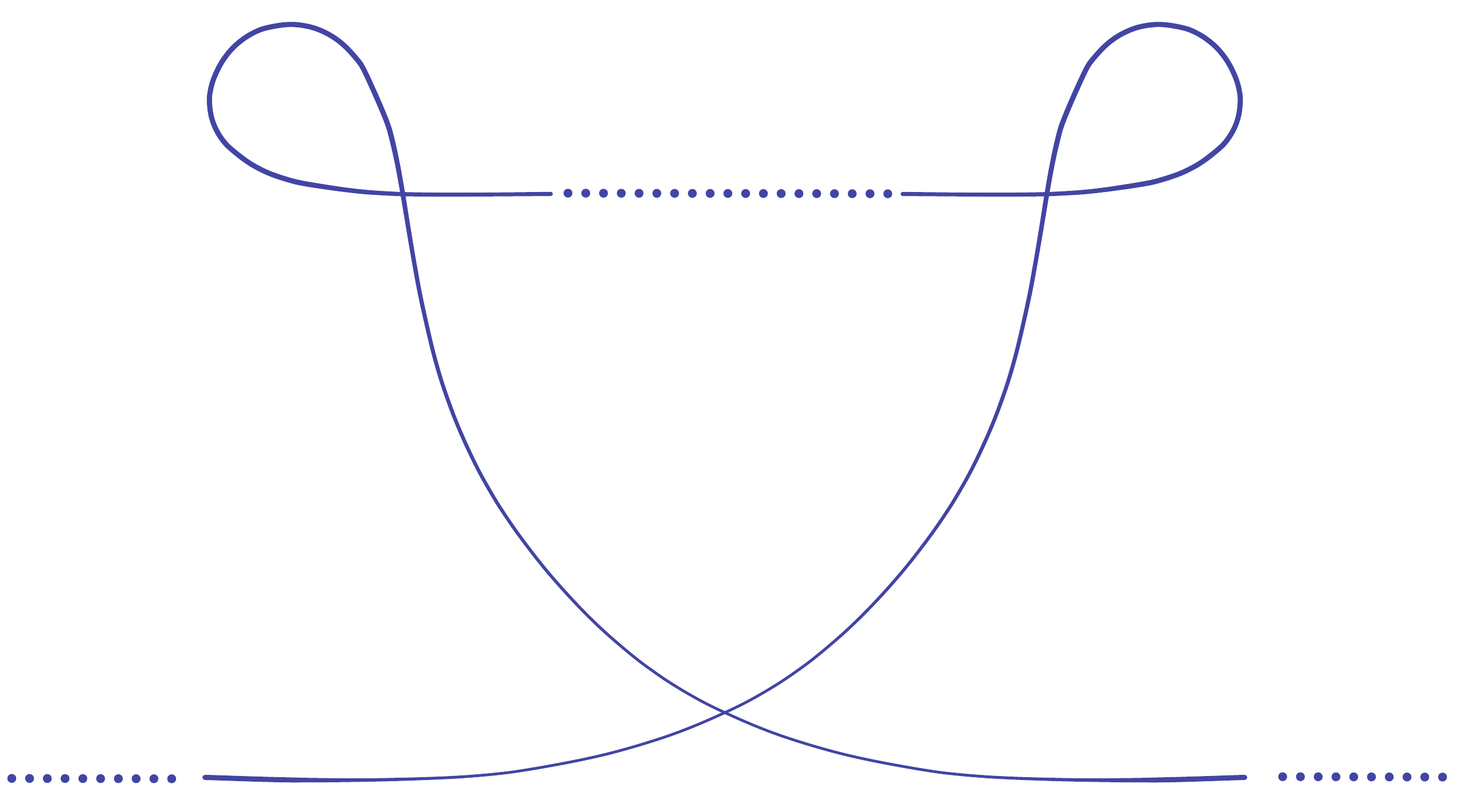}};
			\node[anchor=south] at (0,0) {\includegraphics[width=0.9\textwidth]{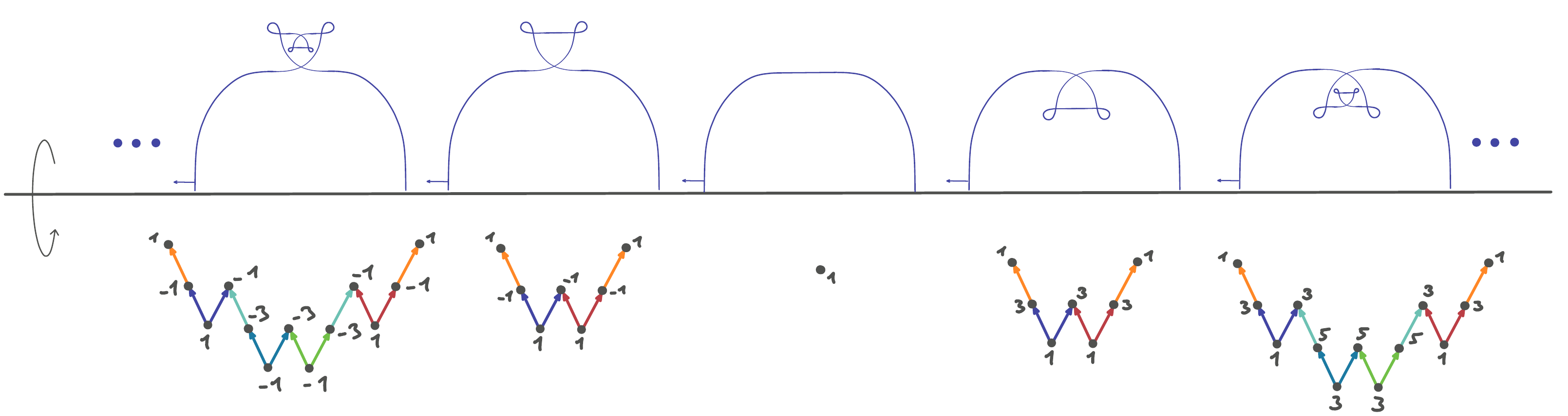}};
			
			\begin{scope}[shift={(0.22*\dd,0)}]
				\foreach \k/\label in {-2.2/$F=2e_1-e_{\m 3}$, -1/$2e_1-e_{\m 1}$, 0/$2e_1-e_1$, 1/$2e_1-e_3$, 2/$ 2e_1-e_5 $} {
					\node[anchor=north] at (\k*\dd,0) {\small \label};
				}
			\end{scope}
		\end{tikzpicture}
		\caption{The fundamental piece to construct $ g_k $ and some examples for $ k=-3,-1,1,3,5 $.}\label{fig_F2e1-ek}
	\end{figure}

	In the following, we show that the image of $ F $ is, in some sense, an affine hyperplane in $ l^1(\Z) $, namely
	\begin{equation*}\label{eq_U}
		U:=\left\{ (h_k)_{k\in\Z} \in \Z^\Z \ \middle| \ h_{2k} = 0 \text{ for all } k\in\Z \text{ and }\sum_{k\in\Z} |h_k|<\infty \text{ and } \sum_{k\in\Z} h_k = 1 \right\}.
	\end{equation*}
	
	\begin{lem}\label{lem_UinImF}
		For every $ h\in U $, there exists a surface of revolution $ i_h\in \Imm_{<3}(\S^2,\R^3) $ with $ F(i_h)=h $.
	\end{lem}
	\begin{proof}
		Let us first describe a special connected sum for two surfaces of revolution whose oriented normal at the poles points into the unbounded component of $ \R^3\setminus f(\S^2) $. Consider the generating curves for the two surfaces and replace neighborhoods of the south pole of the first and the north pole of the second, by a joining neck (see Figure~\ref{fig_FSORGlue}).
		\begin{figure}[H]
			\centering
			\begin{tikzpicture}
				\def\dd{0.04*\textwidth};
				
				\node[anchor=south] at (-0.25*\dd,0) {		\includegraphics[width=0.6\textwidth]{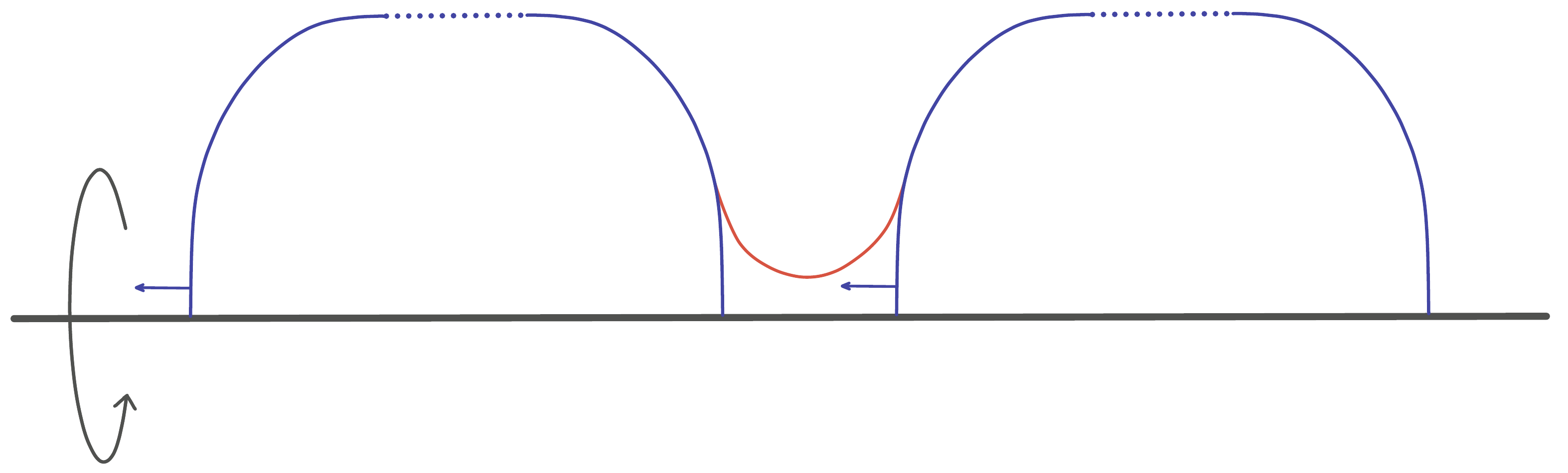}};
				
				\node[anchor=north] at (-0.8*\dd,1.5*\dd) {\tiny \shortstack{ south \\ pole}};
				\node[anchor=north] at (0.8*\dd,1.5*\dd) {\tiny \shortstack{ north \\ pole}};
			\end{tikzpicture}

			\caption{Connect two spheres of revolution along their poles.}\label{fig_FSORGlue}
		\end{figure}
		Let $ f  \#  g $ describe this connected sum. Then the double point tree of $ f \# g $ is just both double point trees of $ f $ and $ g $ glued together along extremal vertices (representing one of the poles) of topological degree $ \delta=1 $. Therefore, we have
		\begin{equation}\label{eq_FSORGlue}
			F(f \#  g) = F(f) + F(g) - e_1.
		\end{equation}
		Let $ h\in U $. Since $ h $ has finitely many nonzero entries whose sum is equal to one, there exists $ n:=\sum\limits_{k\in\Z} \max\{0,h_k\}$ and functions $ a:\{1,\dots,n\}\to\Z,\ b:\{1,\dots,n-1\}\to\Z $ with
		$$ h=\sum_{l=1}^{n} e_{a(l)} -  \sum_{l=1}^{n-1} e_{b(l)}. $$
		We then construct a surface of revolution from gluing together the surfaces $ f_k $ and $ g_k $ from Examples~\ref{ex_Fek} and~\ref{ex_F2e1-ek}. We construct the connected sum
		\begin{align*}
			i_h:=f_{a(1)}  \#  g_{b(1)}  \#  f_{a(2)}  \#  g_{b(2)}  \#  \dots  \#  f_{a(n-1)}  \#  g_{b(n-1)}  \#  f_{a(n)}.
		\end{align*}
		We obtain with \eqref{eq_FSORGlue}
		\begin{align*}
			F(i_h) &= \sum_{l=1}^{n} F(f_{a(l)}) +  \sum_{l=1}^{n-1} F(g_{b(l)}) -(2n-2)e_1\\
			&= \sum_{l=1}^{n} e_{a(l)} -  \sum_{l=1}^{n-1} e_{b(l)} = h. \qedhere
		\end{align*}
	\end{proof}
	\begin{proof}[Proof of Theorem~\ref{thm_imageF}]
		Together with Lemma~\ref{lem_UinImF}, it remains to show $ \im(F)\subseteq U $. 
		Let $ f\in \Imm_{<3}(\S^2,\R^3) $ and $ h:=F(f) $. By Lemma~\ref{lem_delta-f}~\ref{lem_delta-f_odd}, we have $ h_{2k}=0 $ for all $ k\in\Z $. Furthermore, $ f $ has only finitely many double point curves, hence $ G_f $ is a finite tree and thus $ \sum_{k\in\Z} |h_k|<\infty $. The last condition
		$$ 1 = \sum_{k\in\Z} h_k = \sum_{v\in V_f} (1-\deg^-(v)) $$ 
		is satisfied by every directed tree since 
		\begin{equation*}
			\sum_{v\in V_f} (1-\deg(v)^-) = |V_f| - \sum_{v\in V_f} \deg(v)^- = |V_f|-|E_f| = 1. \qedhere
		\end{equation*}
	\end{proof}

\section*{Acknowledgments}
The author thanks his supervisor Elena Mäder-Baumdicker for continuous support and extensive discussions that significantly contributed to this paper. He is also grateful to Robert Kusner and Tahl Nowik for helpful comments.
The author would like to thank the Hausdorff Institute\footnote{funded by the Deutsche Forschungsgemeinschaft (DFG, German Research Foundation) under Germany's Excellence Strategy – EXC-2047/1 – 390685813} in Bonn for a welcoming and productive stay, where part of this work was completed. This work is part of the author's PhD thesis.

\phantomsection\addcontentsline{toc}{section}{References}  
\bibliography{/home/seidel/Nextcloud/MyStuff/MyLit/000_bibs/total.bib}

\end{document}